\documentclass{amsart}

\usepackage{geometry}
\usepackage{amsthm,amsmath,amsfonts,amssymb,amscd}
\usepackage{mathrsfs}
\usepackage{xypic}
\usepackage[pdftex]{graphicx}
\usepackage{multicol}
\usepackage{tikz}
\usepackage{tikz-3dplot}
\usepackage{enumitem}

\usepackage{url}
\usepackage[colorlinks=true,linkcolor=blue,urlcolor=blue, citecolor=blue]{hyperref}
\usepackage[draft,footnote]{fixme}
\usepackage{amscd}
\usepackage{xcolor}
\usepackage{ifpdf}
\usepackage{hyperref}
\usepackage{todonotes}

\theoremstyle{plain}
\newtheorem{theorem}{Theorem}[section]
\newtheorem{lemma}[theorem]{Lemma}
\newtheorem{proposition}[theorem]{Proposition}
\newtheorem{corollary}[theorem]{Corollary}

\theoremstyle{definition}
\newtheorem{definition}[theorem]{Definition}
\newtheorem{remark}[theorem]{Remark}

\newtheorem{example}[theorem]{Example}

\numberwithin{equation}{section}

\newcommand \ww{\omega}
\newcommand\bQ{{\mathbb Q}}
\newcommand\CC{{\mathbb C}}
\newcommand\KK{{\mathbb{L}}}
\newcommand\K{{\mathbb K}}
\newcommand\RR{{\mathbb R}}

\newcommand\ZZ{{\mathbb Z}}
\newcommand\NN{{\mathbb N}}
\newcommand \ff {\mathbf{{f}}}

\newcommand\cO{{\mathcal O}}
\newcommand\cT{{\mathcal T}}

\newcommand\fm{{\mathfrak m}}

\newcommand\im{\operatorname{im}}
\newcommand\vol{\operatorname{vol}}
\newcommand\hot{\operatorname{h.\!o.\!t.}}
\newcommand\trop[1]{\mathcal{T}(#1)}
\newcommand{\MV}{\rm{MV}}
\newcommand{\face}{\operatorname{face}}
\newcommand\init{\operatorname{in}}
\newcommand\cl{\operatorname{cl}}

\newcommand \eep {\varepsilon}
\newcommand \val {\operatorname{val}}
\newcommand \conv {\operatorname{conv}}

\newcommand \Jac{\operatorname{Jac}}

\def\<{\langle}
\def\>{\rangle}

\title[Curve valuations and mixed volumes in sparse implicitization]{
Curve valuations and mixed volumes in the
implicitization of rational varieties}

\author[A.~Dickenstein, M.I. Herrero and B.~Mourrain] {
Alicia Dickenstein, Mar\'ia Isabel Herrero and Bernard Mourrain}

\date{\today}

\address{Departamento de Matem\'atica\\
FCEN, Universidad de Buenos Aires e IMAS (UBA--CONICET)\\
Ciudad Universitaria, Pab.I \\
1428 Buenos Aires, Argentina} \email{alidick@dm.uba.ar}
\email{iherrero@dm.uba.ar}

\address{INRIA Sophia Antipolis M\'editerran\'ee\\
AROMATH \\
2004 route des Lucioles, B.P. 93, \\
06902 SOPHIA ANTIPOLIS, FRANCE}
\email{Bernard.Mourrain@inria.fr}


  \thanks{MIH and AD were supported by ANPCyT PICT 2016-0398,
Argentina. AD was also partially supported by UBACYT 20020170100048BA, 
CONICET PIP 11220150100473, Argentina.}

\newcounter{example2}[section]
\newenvironment{example2}[1][]{
\par\medskip \noindent \textbf{Example~\ref{ex:Isa}, continuation.}
\rmfamily}{\medskip}

\newcounter{Hensel}[section]

\begin{document}

\begin{abstract}
  We address the description of the tropicalization of families of
  rational varieties under parametrizations with prescribed support,
  via curve valuations. We recover and extend results by Sturmfels,
  Tevelev and Yu for generic coefficients, considering rational
  parametrizations with non-trivial denominator.  The advantage of our point of view
  is that it can be generalized to deal with non-generic parametrizations.  We
  provide a detailed analysis of the degree of the closed image, 
  based on combinatorial conditions on the relative positions of
  the supports of the polynomials defining the parametrization. 
  We obtain a new formula and finer bounds on the degree, when the
  supports of the polynomials are different. We also present a new
  formula and bounds for the order at the origin in case the closed image is a
  hypersurface. 
\end{abstract}

\maketitle

\section{Introduction} \label{sec:introduction}

A classical question that has thrived research in Computational
Algebraic Geometry is the problem of \emph{implicitization}. 
The aim is to describe the prime ideal of polynomial relations among
the coordinates of a rational map $\ff$. 
 We concentrate on sparse elimination and we consider a family
 of rational maps of the following form:
the input is a field $\K$ of characteristic zero,
$n+1$ finite sets $A_0, \dots, A_n$ of
lattice points in $\ZZ^d$ and  $n+1$ non-zero Laurent
 polynomials in $d$ variables $f_0, \dots, f_n
\in \K[x_1^{\pm 1},\dots, x_d^{\pm 1}]$ supported on these sets.
More
precisely,
\begin{equation}
f_i=\sum_{a\in A_i} c_{i,a} x^a, x=(x_1, \dots, x_d), \,  c_{i,a}
\in \K, \mbox{ for all }  i \in \{0,\dots,n\},\label{eq:2}
\end{equation}
and our parametrization is of the form
\begin{equation}
  \label{eq:def3}
  \ff\colon (\K^*)^d\dashrightarrow
  (\K^*)^n \qquad \ff=\left(\frac{f_1}{f_0}, \dots, \frac{f_n}{f_0}\right).
\end{equation}

As the ideal of  relations among the coordinates of $\ff$ is defined over
any subfield containing the coefficients, we will assume in what follows without loss of generality
 that $\K$ is algebraically
closed. We can think of the implicitization problem as a
moduli question, as the answer differs depending
on the genericity of the coefficients $(c_{i,a})_{i=0}^n$. 
We will consider $\K$ with the trivial valuation. 
In~\cite{STY07}, Sturmfels, Tevelev and Yu approached this
question with tools of tropical geometry, thus exploiting the
combinatorial nature of the problem. Instead of finding the
ideal of relations, we can compute its tropicalization. This is a
rational weighted balanced polyhedral fan in $\RR^n$ that captures
the combinatorics of this ideal. Tropical implicitization
techniques are well suited to study the generic points of this
space via the theory of Geometric
Tropicalization developed by Hacking, Keel and
Tevelev~\cite{Cueto,HKT09}.  These techniques are hard to address
for special choices of coefficients,
since they depend on the process of resolution of
singularities.

We approach this question from a different perspective. Namely, we
go back to the valuative approach to tropical geometry, manifested
in the Fundamental Theorem of Tropical Geometry~\cite{Draisma08,EKL06},
 and we reinterpret tropical
implicitization in this elementary language. The main
advantage of this approach is that we recover in Theorem~\ref{thm:tropset} 
a straightforward generalization of the description
of~\cite[Theorem 2.1]{STY07} for generic parametrizations, while
developing tools that could be further refined to study the non-generic case.  
In fact, Theorem~\ref{thm:tropset} can be deduced from their result,
as we discuss in Section~\ref{sec:mult}.

We also study in detail the degree of the closed image of
$\ff$ (that is, the closure
of the image  of $\ff$) which we denote by $S$, in
case it has dimension $d$. 
When $S$ is a hypersurface (that is, when $d=n-1$),  one can get the direction and the
length of the edges of the Newton polytope $N(H)$ of a defining equation $H$ for $S$
from the description of the cones in the tropicalization
of the image of $\ff$ and their multiplicities 
Then, $N(H)$ can be algorithmically reconstructed, see for instance~\cite{JensenYu}.  
We could then from this information extract the degree $\deg(H)=\deg(S)$.
In the case of lattice polygons, there is a clear description of $N(H)$ in~\cite{DAS}. 
If  $f_0=1$ and all supports are equal, $\deg(S)$ is computed 
in~\cite{SY94} by means of resultants.
In case $f_0$ is a monomial, the description of the Newton polytope of $H$ was also studied 
in general by Esterov and Khovanskii in~\cite{EK08}, where
they develop an elimination theory of polytopes via mixed 
fiber polytopes in the sense of McMullen. This gives the Newton polytope $N(H)$ only 
for generic coefficients, from which $\deg(S)$ can be read 
in this case.  We give a precise description  in Theorem~\ref{thm:degS}.

There is a straightforward upper bound for $\deg(S)$, given
by the lattice volume of the convex hull $(\conv(\cup_{j=0}^n
A_j))$ of the union of the supports (see for instance~\cite{EK03} in case $d=n-1$). The divergence from this
upper bound for generic coefficients is related to the following question. Given
$A_0,\dots, A_n$ and Laurent polynomials $f_0, \dots, f_n$ in $d$ variables with
these supports respectively and generic coefficients, when $d$ generic
linear combinations  $ \ell_{1,0} f_{0}+ \cdots + \ell_{1,n} f_{n},
\dots, \ell_{d,0} f_{0}+ \cdots + \ell_{d,n} f_{n}$ with
generic coefficients $\ell_{i,j}\in \K$, have 
generic coefficients with respect to the support $\cup_{j=0}^n A_j$?
In Theorem~\ref{thm:degS}  we give conditions on the relative
positions of $A_0,\dots, A_n$ under which this upper bound and some
refined ones are attained. 
The conditions are similar to the conditions of monotonicity of the
mixed volume studied in~\cite{BS17} for the case of $n$ sparse polynomials in
$n-1$ variables, but they are in fact different as Example~\ref{ex:different} shows.

When $S=(H=0)$ is a hypersurface, we also compute by means of curve valuations the order at the origin
${\rm ord}_0(S)={\rm ord}_0(H)$.  Again,
if the Newton polytope $N(H)$ is known, this  information can be extracted from it.  
Under our hypotheses,
the origin is never in the image of $\ff$, but it could be in its closure.
We compute ${\rm ord}_0(S)$ in Theorem~\ref{thm:ord} under a condition on the family of
supports that holds in many cases. This order can be greater than one 
(that is, the origin is a singular point of the closure of $S$ in $\K^n$) even for generic coefficients, depending again 
on the relative positions of the supports $A_0,\dots,A_n$.
We also give along the paper examples not satisfying the hypotheses of our results, but
were we can still compute degree and order adapting the arguments in our proofs. 

In Section~\ref{sec:descr-trop-as} we describe the tropicalization of $S$. We prove the basic
Theorem~\ref{thm:trop} and the main Theorem~\ref{thm:tropset}. 
We also show in \S~\ref{sec:example} how this result
and its proof can be used in generic and non-generic cases.  Section~\ref{sec:degr-impl-equat} deals with the
degree computation for any $d$. Our main result there is Theorem~\ref{thm:degS} and in \S~\ref{ssec:suffcond} we
give different conditions for equality in the inequalities in its statement, in case $d=n-1$. Section~\ref{sec:order}
deals with the computation of the order of $S$ at the origin for $d=n-1$. We present our main result 
 Theorem~\ref{thm:ord} and we discuss examples that show how its proof
 can be used to study cases in which the hypotheses are not satisfied.  

\medskip

\noindent{\bf Acknowledgements:} We are very grateful to M. Ang\'elica Cueto, who worked with us at an earlier stage
of this project. In particular, she suggested to us the example we present in Section~\ref{sec:example} and provided
the first picture in Example~\ref{ex:nongeneric}.

\section{The tropicalization of the image of $\ff$}
\label{sec:descr-trop-as}

We fix a family of finite integer sets $A_0, \dots, A_{n} 
\subset \RR^d$ and we take $n+1$ Laurent polynomials $f_0, \dots, f_n$ 
as in~\eqref{eq:2}  with these respective supports and coefficients in $\K$. 
The goal
of this section is to describe the tropicalization of the closure
$\overline{{\rm im}(\ff)}$ of the image of $\ff$ in~\eqref{eq:def3} in case
 the coefficients of $f_0, \dots, f_n$ are generic. 
Since the parametrization is invariant under multiplication of
all the $f_i$'s by the same monomial, that is under a common translation of the supports, 
we will assume all along, without loss of generality, that our
supports $A_0,\dots, A_n$ lie in the positive orthant $(\ZZ_{\geq
0})^d$, that is, that $f_0, \dots, f_n\in \K[x_1,\dots, x_d]$. 

Our main result in this section is Theorem~\ref{thm:tropset}, which
describes the tropicalization of the closure of the image of $\ff$ as a set.  It 
 extends~\cite[Theorem 2.1]{STY07} to the case
of rational functions.
Our proof is built from
Theorem~\ref{thm:trop} and Lemma~\ref{lem:order} below using curve valuations. The main
advantage of this approach is that it only involves elementary
techniques and it hints on how to proceed in case the coefficients fail to be generic. 
We develop this idea in
\S~\ref{sec:example}. In Section~\ref{sec:mult} we recover the associated
multiplicities from the results in~\cite{ST08}.

\subsection{Tropical varieties}\label{sec:TropVar}

In this section, we present some basic definitions and notation.
In particular, we give an introduction to useful statements on
tropical geometry from the valuative perspective and recall the
connection between initial ideals, tropicalizations and power
series solutions to Laurent polynomial equations.
Our main result here is Theorem~\ref{thm:trop}.

We start by recalling the
definition of the tropical variety of an ideal from the point of
view of Gr\"obner theory. We consider an algebraically closed field $\K$ 
of characteristic zero with trivial valuation. As usual, we set
$\K^* = 
\K \setminus \{0\}$ and we call $(\K^*)^d$ the $d$-torus. For any nonzero
Laurent polynomial $g = \sum g_\alpha x^\alpha \in \K[x_1^{\pm 1}, \dots, x_d^{\pm 1}]$ 
and $w \in \RR^n$, we denote
by $\init_w(g)$ the subsum of those terms $g_{\alpha}x^{\alpha}$ in
$g$ for which $g_{\alpha}\neq 0$ and 
$\langle w, \alpha\rangle$ is minimum. 
We denote by $\init_w(I)$ the ideal generated by all the initials $ \init_w(h)$ with respect to $w$ of all nonzero polynomials 
$h\in I$.

 Given an ideal $I \subset \K[x_1^{\pm 1}, \dots, x_d^{\pm 1}]$, 
 the support of the associated {\em tropical variety} or \emph{tropicalization}  $\trop{I}$ of $I$
 is the set:
\begin{equation}\label{eq:tropI}
\trop{I}=\{w \in \RR^{d}\ | \init_{w}(I)\ \mathrm{does\ not\
contain\ any \ monomial}\}.
\end{equation}
Given an algebraic variety $V \subseteq (\K^*)^d$, we set
$\trop{V} = \trop{I(V)}$, where  $I(V)$ denotes the ideal of Laurent polynomials
vanishing on $V$.

By Hilbert's Nullstellensatz,  $\cT(I)$ records all
points $w$ such that the associated variety $V_{\K^*}(\init_w(I))$ is
nonempty. When the ideal $I$ is principal, generated by a Laurent
polynomial $h$, the set $\cT(I)$ consists of all directions $\ww$
for which the initial form $\init_{\ww}(h)$ is not a monomial. Thus,
$\cT (I)$ consists precisely of the codimension one cones in the
inner normal fan of the Newton polytope $N(h)$ of $h$, that is, of the
convex hull of the exponents of the monomials occuring in $h$ with nonzero
coefficient.

Tropical geometry can be
approached from the perspective of valuations, when we think of
them as Bieri-Groves'  sets~\cite{BG84}.
To state explicitly the link between the previous definition~\eqref{eq:tropI}
and the characterization of the tropicalization of  $\trop{I}$ via 
curve valuations, we introduce the algebraically closed field of power series with real
  exponents and {\sl well-ordered} supports
 $\KK=\K\{\!\{\varepsilon^{\RR}\}\!\}$, known as \emph{generalized
    Puiseux series}.
    
An element $\sigma\in \KK^*$ has the form
$\sigma = b_{0}\, \varepsilon^{\alpha_{0}} + \hot(\varepsilon)$,
where $b_{0}\in \K^*$, $\alpha_{0}\in \RR$ and $\hot(\varepsilon)$ is a sum of terms of
the form $b_{i}\, \varepsilon^{\alpha_{i}}$ with $b_i\in \K$ and
$\alpha_{i+1}> \alpha_{i}$ for all $i\ge 0$. We call
$\init(\sigma):= b_{0}\, \varepsilon^{\alpha_{0}}$ its
\emph{initial term} and $\val(\sigma):=\alpha_{0}$ its
\emph{order}. In particular $\val(c)=0$ for all $c \in \K^*$. We
can extend these notions to tuples in $\KK^d$ coordinatewise. 

\medskip

The {\sl Fundamental Theorem of Tropical Geometry}~\cite{Draisma08,EKL06}  
asserts that given an ideal $I$ in the Laurent polynomial ring
  $\K[x_1^{\pm 1}, \dots, x_d^{\pm 1}]$, the tropical variety $\trop{I}$ is the set of
  all valuations of nonzero generalized Puiseux series solutions to $I$:
\begin{equation} \label{eq:FTTG}
\trop{I}= \{ \val(\zeta) \; |\; \zeta \in V_{\KK^*}(I) \}.
\end{equation}
Note that by choosing the field of  generalized Puiseux series $\KK$ which is algebraically closed and complete  
~\cite{Poonen93} with value group $\RR$, it is not
necessary to take closure  on the right side ~\cite{Markwig10}.

The  tropicalization $\trop{I}$ is a rational polyhedral set that can be given
a (non-unique rational polyhedral) fan structure and any such fan can be refined to 
a {\sl tropical fan structure}.  The multiplicity of a maximal cone $\sigma$ in a fan
structure of $\trop{I}$ is defined as the  sum over all minimal associated
primes $P$ of the initial ideal $\rm{in}_w(I)$, of the multiplicities:
$$\rm{mult}(\sigma) = \sum \rm{mult}(P, \rm{in}_w(I)),$$ 
where $w$ is any point in the relative interior of $\sigma$.
These multiplicities satisfy a balancing condition. We refer the reader to~\cite{ST08}
for details.

Consider now  nonzero Laurent polynomials $f_0, \dots, f_n$ and the map $\ff$ from~\eqref{eq:def3}. Denote by
$\ff_{\KK}: (\KK^*)^d \dashrightarrow \KK^n$  the extension of $\ff$ to $\KK$.  
If $I_{\ff}\subset \K[y_1^{\pm1}, \dots, y_n^{\pm 1}]$
is the ideal that defines the variety $\overline{\im \ff}$ and $J_{\ff}\subset \KK[y_1^{\pm1}, \dots, y_n^{\pm1}]$,
 the ideal that defines
the variety $\overline{\im \ff_{\KK}}$, then $J_{\ff} =
I_{\ff}\otimes_\K\KK$. 
Given a set $X$, denote by $\cl(X)$ its closure in the Euclidean
topology. The next result, which is an easy consequence of~\eqref{eq:FTTG}, 
 shows the connection between the
tropicalization of the variety parameterized by the rational map
$\ff$ and the valuations of points in its image.

\begin{theorem}\label{thm:trop} Let $f_0, \dots, f_n$ be nonzero Laurent polynomials in $d$ variables 
with coefficients in $\K$ and
let
 $\ff: \K^d \dashrightarrow
  (\K^*)^n$ be the rational map they define as in~\eqref{eq:def3}.
Then,
\begin{equation}
\cT(I_\ff)= \cl  \{ \val(\ff_{\KK}(\sigma)) \; | \; \sigma \in
(\KK^*)^{d}, \, \ff_{\KK}(\sigma) \in (\KK^*)^{n}\}.\label{eq:1}
\end{equation}
\end{theorem}

\begin{proof}
Let $V= \{ \val(\ff_{\KK}(\sigma)) \; | \; \sigma \in (\KK^*)^{d}, \, \ff_{\KK}(\sigma) \in (\KK^*)^{n}\}$.
The Fundamental Theorem of Tropical Geometry implies that $\val(\ff_{\KK}(\sigma))\in
\trop{I_\ff}$ for every $\sigma \in (\KK^*)^d$ in the preimage
of the torus $(\KK^*)^n$ belonging to the domain of $\ff_\KK$. As tropical varieties are closed in the
Euclidean topology, we have  the inclusion  $\cl(V)\subset \cT(I_\ff)$ in~\eqref{eq:1}.
Moreover, since $V$ is dense in
$\cT(I_\ff)$, we have equality in~\eqref{eq:1} by  Theorem~A in \cite{BG84}.
 Indeed, Bieri-Groves Theorem~A asserts that all maximal
cones in $\trop{I_\ff}$ have dimension equal to $\dim(\overline{\im \ff_{\KK}})$. Let
$Y \subset \overline{\im \ff_{\KK}}$ be
an algebraic variety of codimension $1$ such that
$\overline{\im \ff_{\KK}} \smallsetminus Y\subset \im \ff_{\KK}
\subset (\KK^*)^n$. Then, all maximal cones in $\trop{I(Y)}$ 
have dimension  equal to $\dim (Y) < \dim(\overline{\im \ff_{\KK}})$.
Consequently, $\cT(I_\ff)\setminus \cT(I(Y))$ is dense in $\cT(I_{\ff})$,
as well as $V$.
\end{proof}


\subsection{Defining the cones and the statement of Theorem~\ref{thm:tropset}}\label{ssec:defcones}

We present in Definition~\ref{def:conesC_J}  the
cones that occur in the tropicalization  of the image of $\ff$.
We first need some notation concerning supports and polytopes.
Let $B$ be a finite set in $\ZZ^d$ and $h \in \KK[x_1, \dots,
x_d]$ a polynomial $h = \sum_{q \in B}c_qx^q$ with support $B$ with
$c_q \ne 0$ for all $q \in B$. For any $\alpha\in \RR^d$, we
define
\begin{equation}
  \label{eq:4.2}
  m_\alpha(h) = \min_{q \in B}\{\val(c_q)+\<\alpha, q\>\}
  \mbox{ and } \init_{\alpha}(h) = \sum_{m_\alpha(h) =
  \val(c_q)+\<\alpha, q\>}\init(c_q)x^q.
\end{equation}
In particular, if $h \in \K[x_1, \dots, x_d]$, $m_\alpha(h)$ does
not depend on its coefficients and we have
\begin{equation}
  \label{eq:4.1}
  m_\alpha(B) = m_\alpha(h) = \min_{q \in B}\{\<\alpha, q\>\}
  \mbox{ and } \init_{\alpha}(h) = \sum_{m_\alpha(B) =
  \<\alpha, q\>}c_qx^q,
\end{equation}
as at the beginning of \S~\ref{sec:TropVar}.

Let $Q \subset \RR^d$ be a lattice polytope and  $\alpha \in \RR^d$.
 We will call face$_\alpha(Q)$ the face of $Q$
which has $\alpha$ as an interior normal vector. That is,
\begin{equation}\label{eq:face}
\mbox{face}_\alpha(Q)=\{ q \in Q \ | \
<\alpha,q>=m_\alpha(Q)\}.\end{equation} If $B \subset \ZZ^d$ is a
finite set, ${\rm face}_\alpha(B)$  (called a face of $B$) is obtained as the intersection with $B$ of the face of the convex
hull of $B$ defined by $\alpha$. Thus, for any nonzero Laurent polynomial $g$,
 the Newton polytope of $\init_\alpha(g)$ is
face$_\alpha(N(g))$.

Given a lattice polytope $Q \subset \RR^{d}$, we use the notation $\vol(Q)$
for its normalized volume with respect to the lattice $\ZZ^d$ (so
$\vol(Q)$ equals $d!$ times its Euclidean volume $\vol_e(Q)$). 
Given a family of polytopes $Q_1, \dots,
Q_d \in \ZZ^d$, we denote by $\MV(Q_1, \dots, Q_d)$ their $d$-dimensional
mixed volume
\begin{equation} \label{eq:MV}\MV(Q_1, \dots, Q_d) =
\sum_{I \subset \{1, \dots, d\}}(-1)^{n-\#I}\vol_e(\sum_{j \in
I}Q_j).\end{equation} If $B_1, \dots, B_d$ are finite sets, the
mixed volume $\MV(B_1, \dots, B_d)$ of $B_1, \dots, B_d$ is
defined as the mixed volume of their convex hulls $\conv(B_1),
\dots, \conv(B_d)$. Recall that, by Bernstein's Theorem
(\cite{Ber75}), this is an upper bound for the number of isolated
common zeros in $(\KK^*)^{d}$ of any system of polynomials with
support set $B_1, \dots, B_d$. Moreover, equality holds for
generic polynomials with those supports.

\smallskip

We denote by $[n]_0$  the set $[n]_0=\{0, \dots, n\}$ and by $[n]$ the
set $[n] = \{1, \dots, n\}$.

\begin{definition} \label{def:coherent} Given
$B=(B_1, \dots, B_r)$ a family of finite sets in $\RR^d$
(or convex polytopes) and $F_j$ a face of $B_j$ for all
$j \in [r]$, we will say that $F=(F_1, \dots, F_r)$ is a
coherent collection of faces of $B$ if there exist $\alpha \in
\RR^d$ such that $F_j=\mbox{face}_\alpha(B_j)$ for all $ j \in [r]$.
\end{definition}

\begin{definition}\label{def:adapted} 
A family of finite sets $B_1, \dots, B_r$ in $\ZZ^{d}$
 is said to be essential if for any subset
$J \subset [r]$ of size at most $d$, it holds that
\begin{equation}\label{eq:essential}\mbox{dim}(\sum_{i\in J}
\conv(B_i)) \ge |J|.\end{equation}
With the previous notation,
given a set $J\subset [r]$, we will say that a
coherent collection $F=(F_1, \dots. F_r)$ of faces of $B$ is adapted to $J$ if
$|J|\le d$ and the family $\{F_j\}_{j \in J}$ is essential.  
\end{definition}

Consider again the family of supports
 $A_0, \dots, A_n\subset \ZZ_{\ge0}^d$. Given any $\alpha \in \RR^d$, we define
\begin{equation}
  \label{eq:7}
  \psi_i(\alpha) \, = \, m_{\alpha}(A_i) - m_\alpha(A_0),  \quad 1  \in [n],
\end{equation}
 and we consider the map $
\psi: \RR^d \rightarrow \RR^n$  defined by
\begin{equation}\label{eq:psi}
\psi(\alpha)=(\psi_1(\alpha), \dots, \psi_n(\alpha)).
\end{equation}

We now define the cones on which
the description of $\mathcal{T}(\overline{\im \ff})$ is based.
We use the notation $\mathbf{1}=(1, \dots, 1)\in \RR^n$.
Given a subset $J \subset [n]_0$, we use
the standard notation $(\RR_{\ge 0})^{J}$ to denote the real
vectors $y \in \RR^n$ satisfying $y_j \ge 0$ for all $j\in J$ and
$y_{j}=0$ for $j\not\in J$. 

\begin{definition}\label{def:conesC_J}
Given a set $J \subset [n]_0$ and a coherent collection $F=(F_0,
\dots, F_n)$ of faces of $(A_0, \dots, A_n)$ adapted to $J$, we
define the following cones in $\RR^n$:
\begin{itemize}
\item[*] if $0 \notin J$,
  \begin{equation}
C_J^F:= \{ \psi(\alpha) + u  \ | \ \alpha=0 \mbox{ or }
\mbox{face}_\alpha(A_i) = F_i ,  \, \forall\ i \in [n]_0 \mbox{ and }
u \in (\RR_{\ge 0})^J\},
\end{equation}
\item[*] if $0 \in J$, set $J'= J - \{0\}$ and
\begin{equation}
C_J^F:= \{ \psi(\alpha) + u - \lambda \, \mathbf{1}  \ | \
\alpha=0 \mbox{ or } \mbox{face}_\alpha(A_i)=F_i, \, \forall\ i \in
[n]_0, \  u \in (\RR_{\ge 0})^{J'} \mbox{ and } \lambda \in
\RR_{\geq 0}\}.
\end{equation}
\end{itemize}
\end{definition}

Note that there is a bijection between 
faces $\sum_{j=0}^rF_n$ of
$\sum_{j=0}^n A_j$, and cones of
the inner normal fan of $\conv(\sum_{j=0}^nA_j)$.
Denoting by $C^F$ the cone corresponding to a coherent collection $F$ 
of faces of $(A_0, \dots, A_n)$, we have that 
for every $0 \notin J$, the cone $C_J^F$ can be written as $C_J^F
= \psi(C^F) + (\RR_{\ge 0})^J$ as defined by \cite{STY07}, and for
$0 \in J$, $C_J^F = \psi(C^F) + (\RR_{\ge 0})^J + \RR_{\le 0} \,\mathbf{1}.$

We  now present Theorem~\ref{thm:tropset}.
Our statements are direct generalizations
of \cite[Theorem~2.1]{STY07}, but the proofs are different and
based on curve valuations and give a hint on how to describe with this tool
the tropicalization of the image when the coefficients are not generic.

\begin{theorem} \label{thm:tropset} Consider the rational map $\ff$ as
  in~\eqref{eq:def3} where $f_0, \dots, f_n\in \K[x_1,\dots, x_d]$ have supports $A_0, \dots, A_n$ and generic coefficients.
Then, the following subsets of $\RR^n$ coincide:
\begin{enumerate}[label={(\roman*)},itemindent=1em]
\item \label{thmitm:tropset1} the tropical variety $\cT (\overline{\im \ff})$,
\item \label{thmitm:tropset3} the union of the cones $C_J^F$, for all  $J \subseteq [n]_0$ such that
$|J| \leq d$ and all coherent collections $F$ of faces of $(A_0,
\dots, A_n)$ adapted to $J$.
\end{enumerate}
\end{theorem}

The proof will be given in \S~\ref{ssec:proof}.

Denote by $P_i$ the convex hull of $A_i$  for any $i \in [n]_0$.
 For instance, in the unmixed case when all  $P_i$  are
  dilations of a fixed polytope $P\in \RR^d$, that is, when there exists
  positive integers $d_0,\dots, d_n$, such that $P_i=d_i P$ for
  all $i \in [n]_0$~\cite[Section 3.2]{STY07},
 we have that
  $$\psi(\alpha)=m_{\alpha}(P)(d_1-d_0, \dots, d_n-d_0).$$
The image of $\psi$ varies as in \cite{STY07} according to whether
$P$ contains the origin or not.
When $P$ contains the origin, the image of $\psi$ is
$\RR_{\le0}(d_1-d_0,\dots, d_n-d_0)$ and the tropical variety $\cT
(\overline{\im \ff})$ is the union of the cones
$\RR_{\le0}(d_1-d_0, \dots, d_n-d_0) + (\RR_{\ge 0})^J$ over all
$|J|=\dim(P)-1$ such that $0 \not\in J$, and $\RR_{\le0}(d_1-d_0,
\dots, d_n-d_0) + (\RR_{\ge 0})  ^{J\backslash\{0\}}+
\RR_{\ge0}(-1, \dots,-1)$ over all $|J|=\dim(P)-1$ such that $0
\in J$.
When $P$ does not contain the origin, the image of $\psi$ is
$\RR(d_1-d_0, \dots, d_n-d_0)$ and $\cT (\overline{\im \ff})$ is
the union of the cones $\RR(d_1-d_0, \dots, d_n-d_0) + (\RR_{\ge
0})^J$ over all $|J|=\dim(P)-1$ such that $0 \not\in J$, and
$\RR(d_1-d_0, \dots, d_n-d_0) + (\RR_{\ge 0})^{J\backslash\{0\}} +
\RR_{\ge0}(-1, \dots,-1)$ over all $|J|=\dim(P)-1$ such that $0
\in J$.

We could use Theorem~\ref{thm:tropset}
to compute the dimension of the variety parameterized
by $\ff$, which equals the maximal dimension of a cone $C^F_J$,
because the tropical variety
$\cT (\overline{\im \ff})$ and $\overline{\im \ff}$ have the same
dimension by~\cite{BG84}.
For instance, let $d=n-1$. If there
exists a subset $J_0\subset [n]_0$ with
  $|J_0|=n-1$ and $C_{J_0} = \cup_F C^F_{J_0} \neq \emptyset$, then $\ff$ parameterizes
  a hypersurface. Note that the reciprocal is not true. For instance,
  consider the rational map $\ff:(\CC^*)^2 \dashrightarrow (\CC^*)^3$ defined by $\ff(x_1,
x_2)=(c_1x_1, c_2x_1, x_2)$,with $c_1, c_2 \neq 0$. Then, $\ff$ parameterizes a hypersurface
but $C_J=\emptyset$ for all $J \ne \emptyset$, in particular those
with cardinal $2$. This is not a contradiction because in this case
dim$(C_\emptyset)=2.$

Theorem~\ref{thm:tropset} could be seen as a consequence of \cite[Theorem
  2.1]{STY07}, which applies to the case when $f_0$ is a
  monomial ($A_0$ consists of a single point).
Consider again finite sets $A_0,\dots,A_n \subset
  \ZZ^d$ and polynomials $f_0, \dots, f_n$ 
with these respective supports such that the closure of the image of the
associated rational map $\ff$ has dimension $d$. 

Let $\rho$ be the homomorphism of tori with its corresponding
linear map $A$ of lattices:
\begin{equation}
 \begin{matrix}\label{eq:rho}\rho: (\K^*)^{n+1} &
\dashrightarrow & (\K^*)^{n} \\ (z_0, \dots, z_n) & \to &
(\frac{z_1}{z_0}, \dots, \frac{z_n}{z_0})\end{matrix} \quad
 \quad \begin{matrix} A: \ZZ^{n+1} & \to & \ZZ^n \\
(w_0, \dots, w_n) & \to & (w_1-w_0, \dots,
w_n-w_0)\end{matrix}.\end{equation}
Denote by $F$ the rational map:
\begin{equation}\label{eq:F} \mathbf{F}: (\K^*)^d \dashrightarrow
(\K^*)^{n+1} \quad \mathbf{F} = (f_0, f_1, \dots,
f_n).\end{equation}
Then, $\ff = \rho \circ \mathbf{F}$ and thus 
$\cT(\overline{\im \ff}) = A(\cT(\overline{\im \mathbf{F}}))$ (see Remark~1.2 in~\cite{ST08}).
The cones of the tropical variety $\cT(\overline{\im \ff})$ can then
be deduced from the description of the tropical variety
$\cT(\overline{\im \mathbf{F}})$. 
As we show in Section~\ref{sec:example},  our direct and simple approach via curve valuations
can be extended to deal with non-generic parametrizations. Also, 
curve valuations are used in Section~\ref{sec:order} to give a formula for the order
at the origin of $S$.

\subsection{Two auxiliary results}

We  first recall Hensel's lemma. 
The valuation ring of $\KK$ is
$\cO=\{\sigma \in \KK \,| \,\val(\sigma)\geq 0\}$, with maximal ideal $\fm$. Here
$\cO^*=\cO\smallsetminus \fm$ denotes the units in $\cO$, i.e.
those elements with valuation $0$.
Given a sequence of polynomials $h=(h_1, \dots, h_r)\in \KK[x_1,
\dots, x_d]$, let us denote by $J_h(x)\in \KK(x)^{d\times r}$ its
associate Jacobian matrix. Observe that the field $\KK$ is
algebraically closed, and so a Henselian field. 
{\sl Hensel's lemma}  asserts the following (see, for example, \cite{Kuhlmann00}). 
Let $h=(h_1, \dots, h_d)\in \cO[x_1, \dots, x_d]$ be a family of
non-zero polynomials. Assume that there exists $b=(b_1, \dots,
b_d)\in \cO^d$ such that
$$\val(h_i(b))>0 \mbox{ for } 1 \in [d]  \mbox{ and }\val(\det
J_h(b))=0.$$ Then, there exists a unique $\sigma 
\in \cO^d$ such that $h_i(\sigma)=0$ and
$\val(\sigma_i-b_i)>0$ for all $i  \in [d]$.

The following lemma for polynomials with nonnegative exponents is straightforward.

\begin{lemma}\label{lem:initialForms} Let $B\subset \ZZ_{\ge
0}^d$ be a finite set and $h \in \KK[x_1, \dots, x_d]$ be a
non-zero polynomial supported on $B$. For any $\sigma \in
(\KK^*)^d$ with $\init(\sigma)=b\,\varepsilon^\alpha =
  (b_1\eep^{\alpha_1}, \dots, b_d\eep^{\alpha_d})$, $b_{i}\in
  \KK^{*}$, $\alpha_{i}\in \RR$, we have
  \begin{equation}
    \label{eq:5}
    h(\sigma)=\init_{\alpha}(h)(b)\,\varepsilon^{m_{\alpha}(h)} + \hot(\varepsilon).
  \end{equation}
\end{lemma}
\noindent Notice that, a priori, we could have
$\init_{\alpha}(h)(b)=0$ so Lemma~\ref{lem:initialForms} implies that
$\val(h(\sigma))\geq m_{\alpha}(h)$, with equality
if and only if $\init_{\alpha}(h)(b)\ne0$.

Lemma~\ref{lem:order} below gives conditions which ensure that
a vector $w$ can be realized as the (coordinatewise) valuation of a point in $im(\ff_{\KK})$. 

\begin{lemma} \label{lem:order} Let $g_1, \dots, g_r$ be non-zero polynomials
  in $\K[x_1,\dots, x_d]$ with $r \leq d$ and call $B_1, \dots B_r$ in $\ZZ_{\ge 0}^d$ 
  their respective supports. Given $\alpha\in \RR^d$, assume
the initial polynomials $\init_{\alpha}(g_1), \dots,
  \init_{\alpha}(g_r)$ have a common zero $b \in (\K^*)^d$
such that their Jacobian matrix $J_{\init_{\alpha}(g)}(b)$
  has maximal rank $r$.
Then,
\begin{enumerate}[label={(\roman*)},itemindent=1em]
\item \label{lemitm:ordera} for each $\ww\in \RR^r$ verifying that
  $\ww_j> m_{\alpha}(B_j)$ for all $j \in [r]$,
  there exist  $\sigma\in (\KK^*)^d$ with $\init(\sigma)=b \eep^\alpha$ such that
  $\val(g_j(\sigma))=\ww_j$ for all $j\in [r]$.
\item \label{lemitm:orderb} if $g_1,\dots,g_r$ have generic coefficients, for each  $\ww\in \RR^r$ verifying that
  $\ww_j\geq m_{\alpha}(B_j)$ for all $j \in [r]$,
  there exist $\sigma\in (\KK^*)^d$ with $\val(\sigma)=\alpha$ such that
  $\val(g_j(\sigma))=\ww_j$ for all $j \in [r]$.
\end{enumerate}
\end{lemma}

\begin{proof}
Write $g_i=\sum_{a \in B_i}c_{i,a}x^a$ and
denote $\overline{g}_i(x)=\eep^{-m_{\alpha}(B_i)}
g_i(\eep^{\alpha}x)$, $\eep^{\alpha}x =(\eep^{\alpha_1}x_1,
\dots, \eep^{\alpha_d}x_d)$. By Lemma~\ref{lem:initialForms}, we
know that $\overline{g}_i(x)\in \cO[x_1,\dots,x_d]$,
$\init_{0}(\overline{g}_i)=\init_{\alpha}(g_i)$ 
for all $i \in  [r]$ and $J_{\init_0(\overline{g})}=
J_{\init_{\alpha}(g)}$. Therefore, we can assume that $\alpha=0$
and $m_{\alpha}(g_i)=0$ with $g_i\in \cO[x_1, \dots, x_d]$ for all
$i\in  [r]$.

To prove item \ref{lemitm:ordera}, under these assumptions on
$\alpha$ and the polynomials $g_1, \dots, g_r$, consider the
polynomials $h_i=g_i-\eep^{w_i}$ for every $ i \in [d]$ 
and $w \in (\RR_{>0})^r$. By Lemma~\ref{lem:initialForms},
$\val(h_i(b))>0$ for all $ i \in [r]$. Since
$\init_0(h_i)=\init_0(\overline{g}_i)$ and
$J_{\init_0(\overline{g})}(b)$ has rank $r$, for any $d-r$ generic
affine linear forms $\ell_1, \dots, \ell_{d-r} \in \K[x_1, \dots, x_d]$
that vanish on $b$, the Jacobian matrix $J_{\init_0(h),l}$ of
the family of polynomials $\init_0(h_1), \dots, \init_0(h_r), \ell_1,
\dots, \ell_{d-r}$ is invertible when evaluated at $b$. Consider now
the Jacobian matrix $J_{h,l}$ of the family of polynomials $h_1,
\dots, h_r, \ell_1, \dots, \ell_{d-r}$. Since
$\det(J_{h,l}(b))=\det(J_{\init_0(h),l}(b))+{\hot}(\eep)$ and
$\det(J_{\init_0(h),l}(b))\in \K^*$, $\val(\det(J_{h,l}(b)))=0$.
Then, by Hensel's Lemma there exist $\sigma \in \cO^n$ such that
$\init(\sigma)=b$ and $h_i(\sigma)=0$ for all $ i \in [r]$, and
so $g_i(\sigma)=\eep^{\ww_i}$.

To prove item \ref{lemitm:orderb}, note that for a
generic point $c \in (\K^*)^d$ we have
$\val(g_i(c))=0$ for all $ i \in [r]$, thus
$\ww=0\in \RR^r$ satisfies the statement. Consider now any $\ww
\in (\RR_{\ge0})^r$. We need to show that there exists a point
$\sigma \in (\KK^*)^d$ with $\val(\sigma)=0$ and such that
$\val(g_i(\sigma))=\ww_i$ for all $ i \in [r]$.
Consider the set $I=\{i \in [r] \ | \ \ww_i>0\}$.
Since  $\init_0(g_1), \dots, \init_0(g_r)$ are generic polynomials
and have a common zero $b \in (\K^*)^d$, the solution set in
$(\K^*)^d$ is an equidimensional variety of dimension $d-r$. If we
denote by $(B_1)_0, \dots, (B_r)_0$ the supports of these initial
polynomials, $\Delta_d=\{0, e_1, \dots, e_d\}$ the set with $e_i$
the $i$-th vector of the canonical basis of $\RR^d$ and
$\Delta_d^{(d-r)}$ when it is repeated $d-r$ times, then the mixed
volume $MV((B_1)_0, \dots, (B_r)_0,\Delta_d^{(d-r)})$ is
a positive integer. By the monotony of the mixed volume,
$MV(\{(B_i)_0\}_{i \in I}, \{(B_i)_0\cup\{0\}\}_{i\not\in
I}, \Delta_d^{(d-r)})$ is also positive. That is, for $\ell_1, \dots,
\ell_{d-r}$ generic affine linear forms in $\K[x_1, \dots, x_d]$ and
$\{\lambda_i\}_{i \not\in I}\subset \K^*$ generic numbers, 
the polynomials $\init_0({h}_1), \dots,
\init_0({h}_{d})$ have an isolated simple common zero
$c\in(\K^*)^d$, where
$${h}_i(x)=\begin{cases}{g}_i(x) \mbox{ for
every } i \in I \\ {g}_i(x)+\lambda_i \mbox{ for every }
i \not\in I, i \le r \\ \ell_{i-r}(x) \mbox{ for every } r < i \le
d\end{cases}.$$ Consider now the polynomials $h'_1, \dots, h'_d \in
\cO[x_1, \dots, x_d]$ defined as
$h'_i(x)={h}_i(x)-\eep^{\ww_i}$ if $i \in I$ and
$h'_i(x)={h}_i(x)$ otherwise. Since $c\in
(\K^*)^d \subset \cO^d$ satisfies $\val(h'_i(c))>0$ for
all $  i \in [d]$ and
$J_{\init_0(h')}(c)=J_{\init_0({g}),l}(c)$,
then $\val(J_{h'}(c))=0$. Then, by Hensel's Lemma
there exists $\sigma$ such that $\init(\sigma)= c$ and
$h'_i(\sigma)=0$ for all $ i \in [r]$, or equivalently,
$g_i(\sigma)=\eep^{w_i}$ for all $\ww_i>0$ and
$g_i(\sigma)=-\lambda_i\in \K^*$ for all $ \ww_i=0$.
\end{proof}

\subsection{The proof of Theorem~\ref{thm:tropset}}\label{ssec:proof}

We first show that $\cT (I)$ is contained on the union of
the cones $C_J^F$. We consider for any $J \subseteq [n]_0$ the cone
$C_J = \cup_F C^F_J$, where the union is over all coherent collections $F$ of faces of $(A_0,\dots, A_n)$ 
that are adapted to $J$.\footnote{Unlike the cones $C_J^F$, the cones $C_J$ are not
necessarily convex.}
Then, 
\begin{itemize}
\item[*] if $0 \notin J$,
\[ C_J:= \{ \psi(\alpha) + u, \, \alpha \in \tau_J, \mbox{ and } u \in (\RR_{\ge 0})^J\},\]
\item[*] if $0 \in J$, set $J'= J - \{0\}$ and
\[ C_J:= \{ \psi(\alpha) + u - \lambda \, {\bf 1}, \, \alpha \in \tau_J,
 u \in (\RR_{\ge 0})^{J'}, \mbox{ and } \lambda \in \RR_{\geq 0}\},\]
\end{itemize}
where 
$\tau_J:= \{ \alpha \in \RR^d / \dim(\sum_{j \in K}\face_\alpha
(\conv(A_j))) \ge |K|, \forall \, K \subseteq J\}.$

 By the Fundamental Theorem of Tropical Geometry, it suffices then 
 to show that $\val(\ff_{\KK}(\sigma))$ lies in one of
  the cones $C_J$ for $\sigma$  in $(\KK^*)^d\backslash
  \{\prod_{i=0}^nf_i=0\}$. Let $\alpha = \val(\sigma)$, with $\sigma =
  (b_{1}\,\eep^{\alpha_1}+ \hot(\eep),\dots, b_{d}\,\eep^{\alpha_d}+
  \hot(\eep))$ and $b=(b_{1}, \dots, b_{d})\in (\K^{*})^d$.
  Consider the initial forms $\init_{\alpha}(f_i)$. If
  $\init_{\alpha}(f_0)(b)\neq 0$, take $J$ as the subset of $[n]$
  of those indices for which $\init_{\alpha}(f_i)(b)=0$. 
    
  Since
  $b\in (\K^*)^n$ is a common zero of the generic polynomials
  $\init_{\alpha}(f_i))_{i \in J}$, for all $\alpha \in \tau_J$,
  Lemma~\ref{lem:initialForms} implies that 
  $\val(f_i(\sigma))= m_i(\alpha)$ for all $i \in [n]_{0} \backslash J$ and
  $\val(f_i(\sigma))> m_i(\alpha)$ for all $i \in J$, so $\val(\ff_{\KK}(\sigma))=
  \psi(\alpha)+u \in C_J$ where $\alpha\in \tau_J$ and $u \in \RR^J$.
  On the contrary, if $\init_{\alpha}(f_0)(b)=0$, this means that
  $\val(f_0(\sigma))=m_0(\alpha)+\lambda > m_0(\alpha)$. Take $J'$
  be the set of indices $j\geq 1$ satisfying $\init_\alpha(f_j)(b)=0$ and $J=J'\cup\{0\}$,
  it follows by the same argument that
  $\val(\ff_{\KK}(\sigma)) \in C_{J}$.

\medskip

Our next task is to prove that the cones $C_J$ lie in $\cT (I)$
if the polynomials $f_i$ have generic coefficients.
Let $\alpha \in
\tau_J$. Since the polynomials $f_i$ are generic, we may assume
that there exists $b\in (\K^*)^d$ with $\init_{\alpha}(f_j)(b)=0$
if and only if $j\in J$, with the additional property that the
Jacobian matrix $\Jac_{(in_\alpha(f_j))_{j\in J}}(b)$ associated
to those initial polynomials has maximal rank. In particular,
$\init_{\alpha}(f_i (b))\neq 0$ for every $i \in [n]_0\backslash
J$ and, by Lemma~\ref{lem:initialForms},
$\val(f_i(\sigma))=m_\alpha(P_i)$ for every $\sigma$ with
$\init(\sigma)=b\,\eep^{\alpha}$.

First, assume $0\notin J$ and consider $\psi(\alpha)+u \in C_J$
where $u \in (\RR_{>0})^J$ and $\ww\in \RR^{n}$ defined as $\ww=
\psi(\alpha)+u+ m_\alpha(P_0)(1, \dots, 1).$
Applying Lemma~\ref{lem:order} to the polynomials (and coordinates
of $\ww$) indexed by $J$, there exists $\sigma \in (\KK^*)^d$ with
$\init(\sigma)=b\eep^{\alpha}$ and $\val(f_i(\sigma))=\ww_i$ for
every $i \in J$. Hence,
$\val(\ff_{\KK}(\sigma))=\ww-m_\alpha(P_0)(1, \dots,
1)=\psi(\alpha)+u$. Since $\cT(I)$ is a closed set, we deduce
that $C_J\subset \cT (I)$.
In case $0\in J$, consider
$\psi(\alpha)+u-\lambda(1, \dots, 1)\in C_J$ where $u \in
(\RR_{>0})^{J'}$ and $\lambda>0$ and $\ww = (\ww_0, \dots, \ww_n)
=(m_\alpha(P_0)+\lambda, m_\alpha(P_1)+u_1, \dots,
m_\alpha(P_n)+u_n)\in \RR^{n+1}$.
Applying Lemma~\ref{lem:order} to the polynomials $\{f_i\}_{i \in
J}$ (and the corresponding coordinates of $\ww$), there exists
$\sigma \in (\KK^*)^d$ with $\init(\sigma)=b\eep^{\alpha}$ and
$\val(f_i(\sigma))=\ww_i$ for all $i \in J$. Then
$\val(\ff_{\KK}(\sigma))= \psi(\alpha)+u-\lambda(1, \dots, 1).$

 \subsection{An example}
 \label{sec:example}

We apply Theorem~\ref{thm:tropset} to the simple Example~\ref{ex:generic} 
with generic polynomials and we then consider in Example~\ref{ex:nongeneric} polynomials
with the same supports but non-generic coefficients. These examples  are addressed in terms
of geometric tropicalization in Examples~4.3 and~5.3 in~\cite{Cueto}.
In both cases, we consider polynomials $f_0,\dots, f_3$ with the following 
respective supports $A_0, \dots, A_3\subset(\ZZ_{\ge0})^2$:
$A_0= \{0\}, A_1 = \{ (2,0), (3,0), (0,2)\}$, $A_2= \{(0,2), (0,3), (2,0)\}$, $A_3 =
\{(1,1), (0,3), (1,2), (2,1), (3,0)\}$. We depict the corresponding Newton polytopes
$P_i=\conv{(A_i)}$ for  $i =1, 2, 3$:
\begin{center}
\begin{tikzpicture}
\draw (0,4) -- (0,0) -- (4,0); \draw (0,3) node {$-$} (0,2) node
{$-$} (0,1) node {$-$} (1,0) node {$|$} (2,0) node {$|$} (3,0)
node {$|$}; \draw (-0.2,1) node {\small{1}} (-0.2,2) node
{\small{2}} (-0.2,3) node {\small{3}}; \draw (1,-0.4) node
{\small{1}} (2,-0.4) node {\small{2}} (3,-0.4) node {\small{3}};
\draw[fill, red] (0,2) -- (2,0) -- (3,0) -- (0,2); \draw (1.8,1.8)
node {$P_1$};
\end{tikzpicture}
\begin{tikzpicture}
\draw (0,4) -- (0,0) -- (4,0); \draw (0,3) node {$-$} (0,2) node
{$-$} (0,1) node {$-$} (1,0) node {$|$} (2,0) node {$|$} (3,0)
node {$|$}; \draw (-0.2,1) node {\small{1}} (-0.2,2) node
{\small{2}} (-0.2,3) node {\small{3}}; \draw (1,-0.4) node
{\small{1}} (2,-0.4) node {\small{2}} (3,-0.4) node {\small{3}};
\draw[fill, green] (2,0) -- (0,2) -- (0,3) -- (2,0); \draw
(1.8,1.8) node {$P_2$};
\end{tikzpicture}
\begin{tikzpicture}
\draw (0,4) -- (0,0) -- (4,0); \draw (0,3) node {$-$} (0,2) node
{$-$} (0,1) node {$-$} (1,0) node {$|$} (2,0) node {$|$} (3,0)
node {$|$}; \draw (-0.2,1) node {\small{1}} (-0.2,2) node
{\small{2}} (-0.2,3) node {\small{3}}; \draw (1,-0.4) node
{\small{1}} (2,-0.4) node {\small{2}} (3,-0.4) node {\small{3}};
\draw[fill, blue] (3,0) -- (1,1) -- (0,3) -- (3,0); \draw
(1.8,1.8) node {$P_3$};
\end{tikzpicture}
 \end{center}

\begin{example}[Generic coefficients] \label{ex:generic} Consider
the parametrization $\mathbf{f}:(\CC^*)^2 \dashrightarrow
(\CC^*)^3$ defined by the polynomials
 \begin{eqnarray*}
 f_{0} &=& 1\\
 f_{1} &=& a_1x_1^{2} + b_1x_2^{2} + x_1^{3}\\
 f_{2} &=& a_2x_2^{2} + b_2x_1^{2} + x_2^{3}\\
 f_{3} &=& x_1x_2 + a_3x_1^3 + b_3x_1^2x_2 + c_3x_1x_2^2 + d_3x_2^3,\\
 \end{eqnarray*}
where $a_1, b_1, a_2 ,b_2, a_3,b_3, c_3, d_3$
are generic nonzero elements of $\K$.
Applying Theorem~\ref{thm:tropset}, we can check that:
\begin{itemize}
\item For $J\subset [3]_0$ with at least 3 elements, there
is no possible coherent collection of faces adapted to $J$.
\item For $\#J=2$ the only coherent collection of faces adapted to
$J$ is $(A_0, A_1, A_2, A_3)$ and the associated cones are:
$C_1=\mathbb{R}_{\ge 0}r_1 + \mathbb{R}_{\ge 0}r_2, 
C_2=\mathbb{R}_{\ge 0}r_1 + \mathbb{R}_{\ge 0}r_3,
C_3=\mathbb{R}_{\ge 0}r_2 + \mathbb{R}_{\ge 0}r_3$,
 defined by the
rays $r_1=(1,0,0),\ r_2=(0,1,0)$ and $r_3=(0,0,1)$.
\item For $\#J=1$ we obtain the cones
$C_4=\mathbb{R}_{\ge 0}r_3 + \mathbb{R}_{\ge 0}r_4, 
C_5=\mathbb{R}_{\ge 0}r_1 + \mathbb{R}_{\ge 0}r_5, 
C_6=\mathbb{R}_{\ge 0}r_2 + \mathbb{R}_{\ge 0}r_5$,  where we
define the rays $r_4=(-1,-1,-1) \mbox{ and } r_5=(1,1,1),$ and
$C_{a,1} = \mathbb{R}_{\ge 0}r_1 + \mathbb{R}_{\ge 0}r_6, 
C_{a,2} = \mathbb{R}_{\ge 0}r_2 + \mathbb{R}_{\ge 0}r_7, 
C_{a,3}=\mathbb{R}_{\ge 0}r_3 + \mathbb{R}_{\ge 0}r_8$, where $r_6
= (-2,-3,-3), \ r_7 = (-3,-2,-3)$ and $r_8 = (2,2,3)$.
\item Finally, for $J=\emptyset$, the image of $\psi$ give us the cones
$C_{b,1} = \mathbb{R}_{\ge 0}r_4 + \mathbb{R}_{\ge 0}r_6,
 C_{b,2} = \mathbb{R}_{\ge 0}r_4 + \mathbb{R}_{\ge 0}r_7$.
Here, the union of cones $C_{a,i}\cup C_{b,i}=C_{i+6}$ for 
$i=1, 2, 3$ yields the cones
$C_7=\mathbb{R}_{\ge 0}r_1 + \mathbb{R}_{\ge 0}r_4 , 
C_8=\mathbb{R}_{\ge 0}r_2 + \mathbb{R}_{\ge 0}r_4, 
C_9=\mathbb{R}_{\ge 0}r_3 + \mathbb{R}_{\ge 0}r_5$.
\end{itemize}
\begin{figure}[h!]
 \includegraphics[width=10cm]{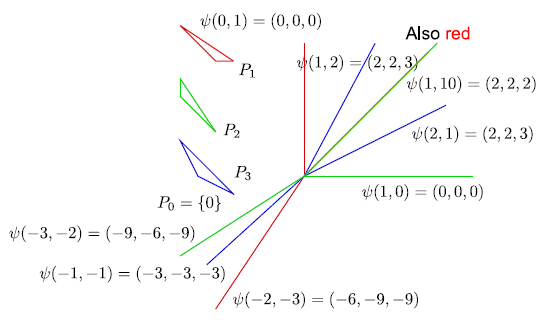}
  \caption{The image by $\psi$ of the cones spanned by $\{(-3, 2),(0, 1)\},$
 $\{(0, 1), (1, 2)\},$ $\{(2, 1), (1, 0)\},$
 $ \{(1, 0), (-2,-3)\}$ is one dimensional.}
 \end{figure}
In fact, the Newton polytope of a (reduced) polynomial $H$ that defines the
closure of the image of $\ff$ is a truncated simplex with vertices
$\{(4,0,0),(0,4,0),(0,0,4),(9,0,0),(0,9,0),(0,0,9)\}$. The cones $C_1$ to $C_9$
correspond to the $9$ edges of $N(H)$. We explain how to compute a priori
the multiplicities (the lengths of the edges) in \S~\ref{sec:mult}.
\end{example}

\begin{example}[Non-generic coefficients]\label{ex:nongeneric}
Consider now the rational map $\ff:(\CC^*)^2 \dashrightarrow
(\CC^*)^3$ associated to the polynomials
 \begin{eqnarray*}
 f_{0} &=& 1\\
 f_{1} &=& x_1^{2} - x_2^{2} - x_1^{3}\\
 f_{2} &=& x_2^{2} - x_1^{2} - x_2^{3}\\
 f_{3} &=& {4x_1x_2} - (x_1 + x_2)^{3}.\\
 \end{eqnarray*}
The union of the zeros of these
$f_i$'s (that is, the base locus of $\ff$) has a highly singular point at the origin: 
\begin{center}
\includegraphics[width=7cm]{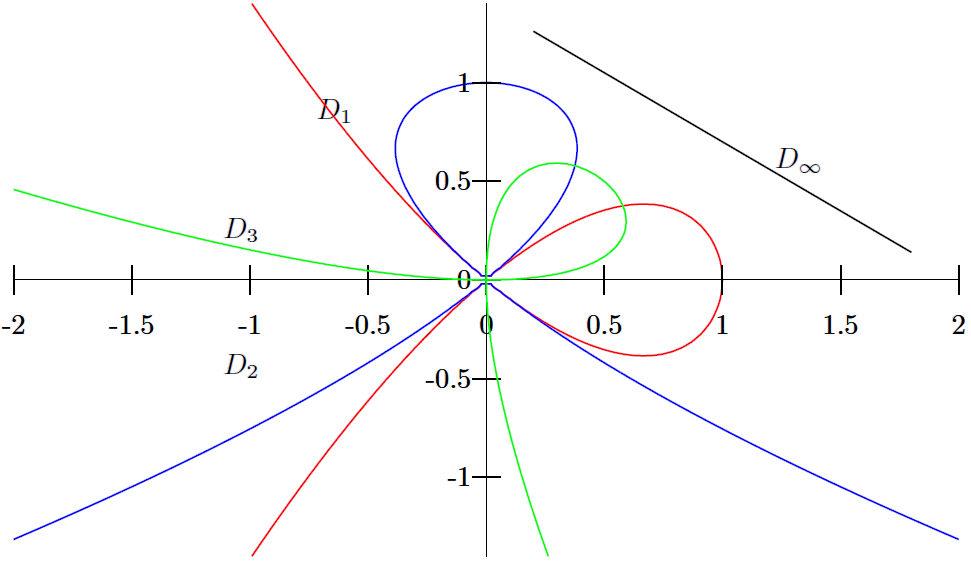}
\end{center}
Let $H$ be a (reduced) polynomial that defines the hypersurface
parameterized by $\ff$. Its Newton polytope $N(H)$ is the
pyramid in $(\ZZ_{\ge 0})^3$ with equation $X+Y+Z\le 9$ cut by the
two half-spaces $3\, X + 3\, Y + 2\, Z \ge 12$ and $2\, X + 2\, Y +  \, Z \ge 7 $.
 Note that it is contained in the convex hull $P$ of the support
 in the generic case in Example~\ref{ex:generic}. 
 We depict $N(H)$ on the right and $P$ on the left:
 \begin{center}
\includegraphics[width=4cm]{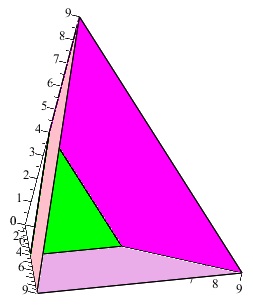} \quad \quad \quad
\includegraphics[width=4cm]{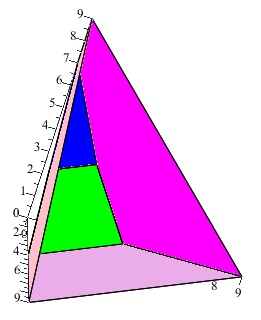}
\end{center}
The tropicalization $\cT(H)$, instead, is not equal to the union of
cones of Theorem~\ref{thm:tropset}, but bigger. However, we can describe some
of the cones in $\trop{H}$ by considering the approach in the
proof of Theorem~\ref{thm:tropset} and using it as a recipe to
find the tropicalization.

In Theorem~\ref{thm:tropset}, to prove that a cone $C_J$ lies in $\mathcal{T}(I)$, 
the genericity condition is used to ensure that there is a common zero $b \in (\CC^*)^{n-1}$ of
$\init_\alpha(f_j)$ if and only if $j\in J$ and that it satisfies
all hypotheses from Lemma~\ref{lem:order}. But this Lemma does not
require that the polynomials are generic. Using this idea, we
consider for all possible $J$ the common zeros of
$\{\init_\alpha(f_j)\}_{j \in J}$ that are not zeros of any
polynomial in $\{\init_\alpha(f_j)\}_{j \not\in J}$.
If either $0 \in J$ or $\#J\ge 3$, then there is no possible
$\alpha$ such that $\{\init_\alpha(f_j)\}_{j \in J}$ have common
zeros. 
If $J=\emptyset$, for all $\alpha \in \RR^2$ there are points $b
\in (\CC^*)^2$ that are not zeros of any $\init_\alpha(f_j)$.
Then, as in the generic case (see Example~\ref{ex:generic}),
\begin{equation}\label{eqref:exC_{b,i}}
\psi(\RR^2)= C_{b,1}\cup C_{b,2}\cup C_{b,3}
\end{equation}
and this set is included in $\mathcal{T}(I)$.
If $J {\in} \{\{1,3\},\{2,3\}\}$ the only vector $\alpha \in \RR^2$ for
which there are common zeros of $\{\init_\alpha(f_j)\}_{j \in J}$ that are not
zeros of {any} $\init_\alpha(f_j)$ for all $j \not\in J$ or of
$\det(\mathcal{J}_{\{\init_\alpha(f_j)\}_{j \in J}})$ is $\alpha =
(0,0)$.  {For $J =\{1,2\}$, a possible vector is $\alpha = (0,0)$ (although is not the only one).}
Then, ${C_1\cup C_2 \cup C_3\subset \mathcal{T}(I)}$.
These 3 are in fact the cones associated to the edges of ${N(H)}$ lying in the
coordinate axes.

When looking at $J{=}
\{1,2\}$, we also need to consider $\alpha=(1,1)$. There is no
zero of $\init_{(1,1)}(f_1)$ that is not a zero of
$\init_{(1,1)}(f_2)$. 
Consider $\sigma = (\varepsilon^v(1+s_1), \varepsilon^v(1+t_1))$,
where $v>0, (s_1,t_1)\in \KK^2$ and $\val(s_1),\val(t_1)> 0$. Then
\begin{eqnarray*}
 f_{1}(\sigma) &=&  \varepsilon^{2v}( 2 s_{1}- 2 t_{1} + s_{1}^{2} - t_{1}^{2} - \varepsilon^v (1+s_{1})^{3})\\
 f_{2}(\sigma) &=& -\varepsilon^{2v}( 2 s_{1}- 2 t_{1} + s_{1}^{2} - t_{1}^{2} + \varepsilon^v (1+t_{1})^{3})\\
 f_{3}(\sigma) &=&  \varepsilon^{2v}( 4 (1+s_{1})(1+t_{1}) - \varepsilon^v (2+s_{1}+t_{1})^{3}).
 \end{eqnarray*}%
If $2 s_{1}- 2 t_{1} + s_{1}^{2} - t_{1}^{2}\ne 0$ and  we
choose $s_1=a\varepsilon^\omega + s_2,t_1= b\varepsilon^\omega +
t_2$ with $0 < \omega < v= \omega + v'$ with $v'>0$, then
$\val(f_{1}(\sigma)) =  \val(f_{2}(\sigma))= 2v' +
3\omega $ and $\val(f_{3}(\sigma))=2v'+ 2\omega.$ Thus,
 the cone $\RR_{\ge 0}(1,1,1) +
\RR_{\ge0}(3,3,2)$ is contained in $\mathcal{T}(I)$. The union of this cone
with ${C_9\subset \mathcal{T}(I)}$ gives the entire
cone corresponding to the edge $\overline{(4,0,0)(0,4,0)}$.
If $J = \{1\}$, we analyze those $\alpha$ such that face$_\alpha(A_1)$
has dimension 1. 
When $\alpha = (0,1)$ we obtain the cone $\RR_{\ge 0}(1,0,0)
\subset C_1$ which we already have.
When $\alpha = (-2,-3)$ there are zeros of
$\init_{(-2,-3)}(f_1) = x_2^2+x_1^3$ that are not zeros of
$\init_{(-2,-3)}(f_2) = \init_{(-2,-3)}(f_3) = x_2^3$ or
$\mathcal{J}_{\init_{(-2,-3)}(f_1)} = (3x_1^2 \ 2x_2)$ and we get
the cone $C_{a,1}$. Recall that (with the notation of
Example~\ref{ex:generic}) $C_{a,i}\cup C_{b,i} = C_{6+i}$.
Interestingly and as in the generic case, since we already
obtained in~\eqref{eqref:exC_{b,i}} the cone $C_{b,1}$ we have
that $C_{a,1}\cup C_{b,1} = {C_7\subset \mathcal{T}(I)}$. In
the same way, for $J=\{2\}$ we obtain that
${C_8\subset \mathcal{T}(I)}$.  These cones correspond to
the edges $[(9,0,0)(0,0,9)]$  and $[(0,9,0)(0,0,9)]$ of $N(H)$.
The set $J=\{3\}$ is  a bit more complicated to analyze. If we consider all faces
face$_\alpha(A_1)$ with dimension 1, both $\alpha=(1,2)$ and
$\alpha=(2,1)$ gives the cone $C_{a,3}$. Via the cones obtained
in~\eqref{eqref:exC_{b,i}}, $C_9\subset \mathcal{T}(I)$. As we
mentioned before, joining $C_9$ with $\RR_{\ge 0}(1,1,1) +
\RR_{\ge0}(3,3,2)$ we obtain the cone
${\RR_{\ge0}(0,0,1) + \RR_{\ge0}(3,3,2)\subset
\mathcal{T}(I)}$
 associated to the edge $[(4,0,0), (0,4,0)]$.
 
Finer computations are required to find the other five
cones associated to the edges of the two new faces of the Newton
polytope of $H$, as it is necessary to consider higher order terms in $\sigma$. 
\end{example}

\subsection{The multiplicities}\label{sec:mult}

We complete the description of
the tropicalization of the closure of the image of $\ff$ with the
computation of the multiplicities of the maximal cones in
Theorem~\ref{thm:tropset}.

  
Recall that as the variety $\overline{\im \ff}$ is
irreducible of dimension $d$, $\mathcal{T}(\overline{\im \ff})$
is a polyhedral set that can be given the structure of a pure $d$-dimensional fan.
Theorem~\ref{thm:tropset} gives a description of
$\mathcal{T}(\overline{\im \ff})$ as a set, and as in \cite{STY07}, this description does not have
a natural fan structure. However, every maximal dimensional cone $\sigma$ in $\mathcal{T}(\overline{\im \ff})$ has
a multiplicity $\rm{mult}(\sigma) \in \ZZ_{>0}$. 
This multiplicity can be computed at any regular point $w$ in the cone,
by means of \cite[Theorem 3.12]{ST08} and \cite[Theorem 5.1]{ST08}.

Consider again the homomorphism of tori $\rho$ in~\eqref{eq:rho}.
Given a regular point $w$ in a maximal cone $\sigma$ of
$\mathcal{T}(\overline{\im \ff})$ such that $f^{-1}(w)$ is a
finite set of regular points in $\mathcal{T}(\overline{\im
\mathbf{F}})$, $m_w=\rm{mult}(\sigma)$ equals the sum over all $v \in \ff^{-1}(w)$, of
the quantities
\begin{equation*} \label{eq:mult} \frac{1}{\deg(\rho)}m_v\cdot
\rm{index}(\mathbb{L}_w\cap \ZZ^n : A(\mathbb{L}_v\cap \ZZ^d)),
\end{equation*}
where the multiplicity $m_v$ of the regular point $v$ in
$\mathcal{T}(\overline{\im \mathbf{F}})$ is as computed in
\cite[Theorem 5.1]{ST08}.

Consider the map $\Psi^{ST}(\alpha) =
(m_\alpha(A_0), \dots, m_\alpha(A_n))$ defined as in~\eqref{eq:7},
 but for the rational map $\mathbf{F}$). For every
$J \subset [n]_0$ with $|J|\ge d$ and every coherent collection of
faces $F=(F_0, \dots, F_n)$ of $(A_0, \dots, A_n)$ adapted to $J$,
define index$(F,J)$ as the index of the lattice $\Psi^{ST}(C^F\cap
\ZZ^d) + \ZZ^J$ in its saturated lattice, when both have rank $d$.
Otherwise, index$(F,J)=0$.  
With this notation, the multiplicity $m_v$ of the regular point $v$ in
$\mathcal{T}(\overline{\im \mathbf{F}})$  is equal to the sum of
\begin{equation*}
\frac{1}{\deg(\mathbf{F})}\rm{index}(F,J)\cdot\MV(F_i\ | \ i \in
J)
\end{equation*} over all sets $J \subset [n]_0$ and all coherent
collection $F$ of faces of $A_0, \dots, A_n$ adapted to $J$ such
that $ \Psi(C^F) + \RR_{\ge 0}^J $ contains $v$. Here, 
$\MV(F_i\ | \ i \in J)$ denotes the $|J|$-dimensional mixed volume. 

The following is a simple example of the multiplicity computation via
this formula.

\begin{example} \label{ex:mult}
Let $f_0, f_1, f_2 \in \K[x]$ be the polynomials
$$f_0 = x^3 + 3x, \ \ f_1 = x^5 + 5x^3 \ \mbox{ and } \ f_2 = x^5
+ 22x^3 + 17x.$$ The rational map $\ff$ has degree 2 and the cones
in the tropicalization of its image are $C_{\{1\}} = \psi(\RR_{\ge
0}) = \RR_{\ge 0} (1,0);\ \ C_{\{2\}} = \RR_{\ge 0} (0,1)$ 
and $C_{\{0\}} = \psi(\RR_{\le 0}) = \RR_{\ge 0} (-1,-1).$ The
map $\mathbf{F}$ has degree 1 and the tropicalization of its image
has $5$ cones. The following table shows the multiplicity $m_w$ of a
regular point $w$ in each of them (and the data used to compute it):

{\large\begin{center}
\begin{tabular}{|c|c|c|c|c|} \hline
$ w $                    & $v \in \rho^{-1}(w)$ & $m_v$ &
${\small\rm{index}(\mathbb{L}_w\cap \ZZ^2 : A(\mathbb{L}_v\cap \ZZ^3))}$ & $m_w$ \\
\hline
                         & $ (\lambda,0,0)    $ & $ 2 $ & $ 1 $ &       \\
$ (-\lambda, -\lambda) $ & $                  $ &       &       & $ 2 $ \\
                & $ \frac{-\lambda}{2}(3,5,5) $ & $ 1 $ & $ 2 $ &       \\
                \hline
                         & $ (0,\lambda,0)    $ & $ 2 $ & $ 1 $ &       \\
$ (\lambda,0) $          &                      &       &       & $ 2 $ \\
                 & $ \frac{\lambda}{2}(1,3,1) $ & $ 1 $ & $ 2 $ &       \\
                 \hline
$ (0,\lambda) $          & $ (0,0,\lambda) $ & $ 4 $ & $ 1 $ & $ 2 $ \\
\hline
\end{tabular}.
\end{center}}

For example, take a regular point $w = (\lambda, 0) \in \RR_{\ge
0}^2$. The set $\ff^{-1}(w)$ has two points in $\cT(\overline{\im
\mathbf{F}})$: the point $v_1 = (0, \lambda, 0)$ which is in a cone
of multiplicity 2 and the point $v_2 = (\lambda/2, 3\lambda/2,
\lambda/2)\}$ which lies in a cone of multiplicity 1. Since
$A(0,1,0)=(1,0)$, the index is 1, but since $A(1,3,1) = 2(1,0)$,
the corresponding index is 2. By equation \eqref{eq:mult}, $m_w =
2$.
\end{example}

\section{Degree of rational varieties under sparse parametrizations}
\label{sec:degr-impl-equat}

As before, we consider 
 $n+1$ supports $A_0, \dots, A_n$ which
lie in the nonnegative orthant $(\ZZ_{\ge 0})^{d}$ and generic
polynomials $f_0, \dots, f_n$ with respective supports $A_0,\dots,
A_n$,  such that the rational map from~\eqref{eq:def3}
$$\ff=\left(\frac{f_1}{f_0}, \dots, \frac{f_n}{f_0}\right)$$ is
generically finite and thus rationally parameterizes 
an irreducible variety $S$ of dimension $d$. We give in Theorem~\ref{thm:degS} an explicit formula
for its degree in terms of the supports when the coefficients are generic,
and which is an upper bound in any case. 
Our sharp bound for the degree of $S$ in Theorem~\ref{thm:degS}  for generic 
coefficients is similar to the bound in Proposition~4 in \cite{HJS19}
in case $f_0 = 1$. The upper bound given by the volume of the convex hull of the union of the supports
appears in Section~3 in~\cite{EK03}. We give  conditions 
for equalities in the chain~\eqref{eq:deg} of inequalities in
Theorem~\ref{thm:vol=mindI}, and we will also show in Example~\ref{ex:Isa} that all the inequalities can
be strict.

Consider the open set $U =\{x \in (K^*)^d \, : \, \prod_{i=0}^n  f_i(x) \neq 0 \}$, the incidence variety 
\begin{equation}\label{eq:W}
W = \{(x,y) \in U \times (\K^*)^{} \, : \,  f_0(x) \, y_1 - f_1(x) = \dots = f_0(x) \, y_n - f_n(x) =0\},
\end{equation}
and the projection
\begin{equation}\label{eq:Wmap}
  \pi \colon W \times (\K^*)^{n} \rightarrow
  (\K^*)^n, \qquad \pi(x,y) = y.
\end{equation}
Then,  $S$ coincides with the closure of the image $\pi(W)$. Moreover, this map is finite iff 
$\ff$ is finite and both have the same degree.
Observe that  the polynomials defining $W$ don't have in general generic coefficients with respect to their  
supports even if $f_0, \dots, f_n$ have generic coefficients (so the results in~\cite{EK08} do not directly apply).   
In fact, we compute $\deg(S)$ in Theorem~\ref{thm:degS}
studying the following similar genericity question: given a subset $I$ of $[n]_0$, with cardinality $d$ and  generic coefficients
$(c_{ij}, i=1, \dots d, j=0, \dots n)$,  when do the polynomials
$\sum_{j=0}^n c_{ij} f_j, \, i=1,\dots, d,$
have generic coefficients with respect to their support? If not, which is 
the number of their common zeros in the torus $(\K^*)^d$?

In order to state our main result in this section, we need to introduce two definitions.

\begin{definition} \label{def:widehat}
 Given a subset $A \subset \ZZ^d$ and a natural number $j$,  denote by ${A}^j \subset \ZZ^{d+j}$ the finite subset
\begin{equation}\label{eq:Ahatj}
{A}^j \, = \, \{(\alpha,0) \in \ZZ^{d+j} \ | \ \alpha \in A\} \cup \{e_{d+1},\dots, e_{d+j}\},
\end{equation}
where $0 \in \ZZ^j$ and $e_i$ denotes the $i$-th canonical basis vector.
\end{definition}

We also define new lattice subsets associated to a collection
of finite subsets.
 
 \begin{definition} \label{def:AiJ}
Let $A_0, \dots, A_n \subset \ZZ^d$ with $d \le n$, be a collection of finite lattice subsets.
  Given $I =\{i_1, \dots, i_d\} \subset [n]_0$ with $|I|=d$, for any $j \in [d]$ we denote 
  \begin{equation}\label{eq:AijI}
   A_{i,I} \, = \, A_{i} \cup (\cup_{j \notin I} A_j) \subset \ZZ^d,
  \end{equation}
  and we define $d_I$ as the mixed volume
$d_I =  \MV(A_{i_1,I}, \dots, A_{i_d,I}).$
 \end{definition}
We then have
\begin{theorem}\label{thm:degS}
Consider Laurent polynomials $f_0, \dots, f_n$  in $d$ variables and coefficients in $\K$ with 
respective supports  $A_0, \dots, A_s \subset \ZZ^d$, such that
 the rational map ${\ff}\colon (\K^*)^d\dashrightarrow  (\K^*)^{n}$  
 defined as in~\eqref{eq:def3} by
${\mathbf f}=\left(\frac{f_1}{f_0}, \dots, \frac{f_n}{f_0}\right)$,
is generically finite with $\deg(\mathbf f)=\delta$.
Then,
 \begin{equation}\label{eq:deg}
  \deg(S) \cdot \delta\, \le \, \MV({A}_0^{n+1-d}, \dots, {A}_n^{n+1-d}) \le \, 
{\rm min}_{|I|=d} d_I \le \, {\rm max}_{|I|=d} d_I  \le \, {\rm vol}({\rm conv}(\cup_{j=0}^n
A_j)).
 \end{equation}
Moreover, the left equality holds when $f_0, \dots,f_n$ have generic coefficients.
\end{theorem}
\begin{proof}
As we mentioned, we can assume without loss of generality (possibly after 
a common translation of the supports)  that all exponents are nonnegative.
The degree of $S$ equals the number of intersection points of $S$ with the zero set of  
 $d$ generic linear forms in $n$ variables with coefficients in
$\K$, which we will write as follows:
\begin{equation} \label{eq:ell}
\ell_i=a_{i,1} y_1 + \cdots + a_{i,n}y_n+a_{i,0}, \quad i \in [d].
\end{equation} 
Since $\ell_1, \dots, \ell_{d}$
are generic, we can assume that every intersection point lies in $\ff(U)$, and it is
enough to consider all common zeros of the form $z = \ff(\sigma)$ for
some $\sigma$ in the domain of $\ff$ (that is, such that $f_i(\sigma) \neq 0$ for
any $i \in [n]_0$).  Such $\sigma$ belongs to the variety defined
by the following polynomials $g_1, \dots, g_d \in \K[x_1, \dots,x_d]$:
\begin{equation*}\label{eq:g}
g_i=\sum_{j=0}^na_{i,j}\, f_j, \quad  i \in [d].
\end{equation*}
These polynomials have support $\cup_{i=0}^n A_i$ and so, 
the number of isolated common zeros in the torus is bounded by 
${\rm vol}({\rm conv}(\cup_{j=0}^n
A_j))$ by Bernstein's theorem, with equality in case $g_1, \dots, g_d$
are generic polynomials with this support.
But, depending on the relative positions of $A_0, \dots, A_n$, this
need not be the case.

A first refinement of this bound is given by the following observation.
As the coefficients
$(a_{i,j})_{i \in [n]_0 \atop  j \in [d]}$ are generic, for
every set $I=\{i_1, \dots, i_{d}\}\subset[n]_0$ with $|I|=d$,
 the system $g_1 =
\dots g_{d} = 0$ is equivalent by means of row operations to a
system of the form
\begin{equation} \label{eq:h_{i,I}}\begin{matrix} h_{1}^I &=& f_{i_1}& & &
+\ \sum_{j \notin I} \mu_{1j}^I \, f_{j} &=& 0
\\ && &  \ddots & & \quad \vdots \\
h_{d}^I &=& & & f_{i_{d}}& +\ \sum_{j \notin I}
\mu_{dj}^I \, f_{j} &=& 0,
\end{matrix}
\end{equation}
where the coefficients $\mu_{ij}^I \in\K$ are generic. 
These polynomials $h_{j}^I$ have  supports $A_{i_j,I}$ for any $j \in [d]$.
By the BKK bound, the system \eqref{eq:h_{i,I}} has
at most $ d_I$ isolated
common zeros in $(\K^*)^{d}$ and by the monotony of the mixed volume,
$d_I \le \MV(\cup_{j=0}^n A_j, \dots, \cup_{j=0}^n A_j)) = 
{\rm vol}({\rm conv}(\cup_{j=0}^n
A_j))$.
But still, depending on the original
supports $A_0, \dots, A_n$, $h_{1}^I, \dots, h_{d}^I$ need not have generic coefficients
with those supports, even if $f_{i}$ are generic for their support $A_{i}$.

Take $I = [d]$ and consider the following polynomials 
in $\K[x_1, \dots,
x_{n+1}]$:
\begin{equation} \label{eq:Istar}\begin{matrix}
h^{[d]}_i &=& f_i+\ \sum_{j=d+1}^{n+1} \mu_{ij}^{[d]}  \, x_{j}, &  1 \le i \le d, \\
{h}_j^{[d]} & = &f_{j} - x_n, & d+1 \le j \le n, \\
{h}_{n+1}^{[d]} &= & f_0 - x_{n+1}. &{}
\end{matrix}
\end{equation}
Denote by $U$ the complement in $(\K^*)^d$ of {the union} of the zeros
of $f_0, \dots, f_{n+1}$, that is, the domain of definition of $\ff$.
Clearly, there is a bijection between common zeros in $U \times (\K^*)^{n+1-d}$ of $h_{1}^{[d]}, \dots,
h_{{n+1}}^{[d]}$  and common zeros in $U$ of $g_1,
\dots, g_{d}$, so $\deg(S)\deg(\ff)$ is bounded above by
 the mixed volume of their supports. But, for
generic $f_0, \dots, f_n$ and generic linear forms $\ell_1, \dots, \ell_d$, these new polynomials $h_{1}^{[d]}, \dots,
h_{n+1}^{[d]}$ have generic coefficients. Moreover, they
do not have common zeros in common in $(\K^*)^{n+1}$ with any $f_i$.
Indeed, no $f_i$ with $i=d+1,\dots,n$ can vanish if $(x_{d+1}, \dots, x_{n+1})
 \in (\K^*)^{n+1-d}$. Take any $i \in [d]$, for instance $i=1$.
Then, if $x=(x_1, \dots,x_{n+1})$ is a solution in the torus of system~\eqref{eq:Istar} and $f_1(x_1,\dots, x_d)=0$,  then $x$ is also a
solution of the system of $n+2$ generic polynomials in $n+1$ variables:
\[f_1 = \sum_{j=d+1}^{n+1} \mu_{1j}^{[d]} \,  x_{j} = {h}_2^{[d]} =
\dots = {h}_{n+1}^{[d]} =0,\]
which is a contradiction.
So the number of common
zeros in $U\times (\K^*)^{n+1-d}$ is the mixed volume of their
supports, $\MV_{n+1}({A}_1^{n+1-d},\dots, A_d^{n+1-d}, (A_{d+1},0) \cup \{e_{d+1}\}, 
\dots, (A_n, 0) \cup \{e_n\},  (A_0, 0) \cup \{e_{n+1}\})$,
where $0 \in \ZZ^{n+1-d}$ as in Definition~\ref{def:widehat}.
Note that replacing $(x_{d+1},\dots,x_{n+1})$ by a generic linear combination,
we would get the same number of solutions in the torus, and thus, this mixed volume coincides 
with the mixed volume $\MV({A}_0^{n+1-d}, \dots, {A}_n^{n+1-d})$, as stated in~\eqref{eq:deg},
and the first inequality follows.

To prove that  $\MV({A}_0^{n+1-d}, \dots, {A}_n^{n+1-d}) \le \,{\rm
  min}_{|I|=d} d_I $, observe that $\MV({A}_0^{n+1-d}, \dots,
{A}_n^{n+1-d})$ is the number of solutions of a generic system of Equations
\eqref{eq:Istar}, which coincides with $\deg(S)\,\deg(\ff)$. As
proved using system~\eqref{eq:h_{i,I}}, we have
$\deg(S)\,\deg(\ff)\leq d_{I}$ for all $I\subset [n]$ with 
$|I|=d$, we deduce the second inequality in \eqref{eq:deg}, which
concludes the proof.
\end{proof}

We deduce from the proof of Theorem~\ref{thm:degS} that the generic value  
$\MV({A}_0^{n+1-d}, \dots, {A}_n^{n+1-d})$  of $\deg(S) \cdot \deg(\ff)$ equals 
$\MV_{n+1}({A}_1^{n+1-d},\dots, A_d^{{n}+1-d}, (A_{d+1},0) \cup \{e_{d+1}\}, 
\dots, (A_n, 0) \cup \{e_n\},  (A_0, 0) \cup \{e_{n+1}\})$. Moreover, we could replace
the choice of indices $[d]$ by any subset $I$ of $[n]_0$ of
cardinal $d$.

\begin{example} 
Consider the application
$\mathbf{f}$ given by the generic polynomials from
Example~\ref{ex:generic}. 
Theorem~\ref{thm:degS} proves the known fact that
$\deg(S)=9$. 
Moreover, in this case the first inequality in~\eqref{eq:deg}
$$\deg(S) \cdot \deg(\ff)\, = \, \MV({A}_0^{2}, \dots, {A}_3^{2}) \le \, 
{\rm min}_{|I|=2} d_I$$ 
in the statement is actually an equality. In fact,
we can consider the steps in the proof of Theorem~\ref{thm:degS} to verify that $\deg(S) \cdot \deg(\ff)=d_{\{1,2\}}.$
We will see below that for general supports the equalities in~\eqref{eq:deg} can be strict.
\end{example}

\subsection{Sufficient conditions for equality}\label{ssec:suffcond}

With the notations from Theorem~\ref{thm:degS}, if
$A_0 = A_1 = \dots = A_n=A$ and $f_0,\dots, f_n$ are generic, the
inequalities in~\eqref{eq:deg} are all equalities and $\deg(S) \cdot \deg(\ff)$
equals $\rm{vol}(\conv(A))$. 
But in general, even for generic polynomials, the 
inequalities may be strict, as we show in the next example.

\begin{example}\label{ex:Isa}
Let $d=2, n+1={4}$, so $S$ is a hypersurface in dimension $3$.
Consider the generic polynomials in $\K[x_1,x_2]$:
$$f_0(x_1,x_2) = x_1x_2,\ \
  f_1(x_1,x_2) = a_1x_2 + b_1x_2^2 + c_1x_1x_2,$$
$$f_2(x_1,x_2) = a_2x_1 + b_2x_1^2 + c_2x_1x_2 \ \mbox{ and } \
  f_3(x_1,x_2) = a_3x_1x_2 + b_3 x_1^3x_2 + c_3x_1x_2^3.$$
Let $A_0,A_1,A_2,A_3$ be their supports sets. We draw their
respective convex hulls $P_i$,  $i \in [3]_0$.
\begin{center}
\begin{tikzpicture}
\draw (0,3) -- (0,0) -- (3,0); \draw (0,2) node {$-$} (0,1) node
{$-$} (1,0) node {$|$} (2,0) node {$|$} ; \draw (-0.2,1) node
{\small{1}} (-0.2,2) node {\small{2}} ; \draw (1,-0.4) node
{\small{1}} (2,-0.4) node {\small{2}} ; \draw[fill, red] (1,1)
circle (2pt); \draw (1.3,1.3) node {$P_0$};
\end{tikzpicture}
\begin{tikzpicture}
\draw (0,3) -- (0,0) -- (3,0); \draw (0,2) node {$-$} (0,1) node
{$-$} (1,0) node {$|$} (2,0) node {$|$} ; \draw (-0.2,1) node
{\small{1}} (-0.2,2) node {\small{2}} ; \draw (1,-0.4) node
{\small{1}} (2,-0.4) node {\small{2}} ; \draw[fill, green] (0,1)
-- (1,1) -- (0,2) -- (0,1); \draw[fill] (0,1) circle (2pt) (1,1)
circle (2pt) (0,2) circle (2pt); \draw (1.3,1.3) node {$P_1$};
\end{tikzpicture}
\begin{tikzpicture}
\draw (0,3) -- (0,0) -- (3,0); \draw (0,2) node {$-$} (0,1) node
{$-$} (1,0) node {$|$} (2,0) node {$|$} ; \draw (-0.2,1) node
{\small{1}} (-0.2,2) node {\small{2}} ; \draw (1,-0.4) node
{\small{1}} (2,-0.4) node {\small{2}} ; \draw[fill, blue] (1,0) --
(1,1) -- (2,0) -- (1,0); \draw[fill] (1,0) circle (2pt) (1,1)
circle (2pt) (2,0) circle (2pt); \draw (1.3,1.3) node {$P_2$};
\end{tikzpicture}

\begin{tikzpicture}
\draw (0,4) -- (0,0) -- (4,0); \draw (0,3) node {$-$} (0,2) node
{$-$} (0,1) node {$-$} (1,0) node {$|$} (2,0) node {$|$} (3,0)
node {$|$}; \draw (-0.2,1) node {\small{1}} (-0.2,2) node
{\small{2}} (-0.2,3) node {\small{3}}; \draw (1,-0.4) node
{\small{1}} (2,-0.4) node {\small{2}} (3,-0.4) node {\small{3}};
\draw[fill, magenta] (3,1) -- (1,1) -- (1,3) -- (3,1); \draw[fill]
(3,1) circle (2pt) (1,1) circle (2pt) (1,3) circle (2pt); \draw
(2.3,2.3) node {$P_3$};
\end{tikzpicture}
\begin{tikzpicture}
\draw (0,4) -- (0,0) -- (4,0); \draw (0,3) node {$-$} (0,2) node
{$-$} (0,1) node {$-$} (1,0) node {$|$} (2,0) node {$|$} (3,0)
node {$|$}; \draw (-0.2,1) node {\small{1}} (-0.2,2) node
{\small{2}} (-0.2,3) node {\small{3}}; \draw (1,-0.4) node
{\small{1}} (2,-0.4) node {\small{2}} (3,-0.4) node {\small{3}};
\draw[fill, cyan] (0,1) -- (1,0) -- (2,0) -- (3,1) -- (1,3) --
(0,2) -- (0,1); \draw[fill] (0,1) circle (2pt) (1,0) circle (2pt)
(2,0) circle (2pt) (3,1) circle (2pt) (1,3) circle (2pt) (0,2)
circle (2pt); \draw (3,3) node {$P=\conv(\cup_{i=0}^3A_i)$};
\end{tikzpicture}
\end{center}

The closure $S \subset \K^3$ of the image of the associated rational map
$\ff=\left(\cfrac{f_1}{f_0}, \cfrac{f_2}{f_0},
\cfrac{f_3}{f_0}\right)$ is a hypersurface.
Let $H$ be a reduced polynomial defining $S$.
We observe that the degree of the
map $\ff$ is $1$, and we can compute
$\deg(S) = \deg(H) = \MV({A}^2_0,
\dots, A_3^2) = 5$,
$\vol(\conv(\cup_{j=0}^3A_j))=11$ and
$$ 5 < 6 = d_{\{1,3\}} = d_{\{2,3\}}
   < d_{\{0,3\}} < d_{\{1,2\}} < d_{\{0,1\}} = d_{\{0,2\}}
   = 10 < 11.$$
In fact, there exists a unique monomial in $H$ of maximal degree $5$.
\end{example} 

In what follows we will present conditions on the supports $A_0,\dots,A_n$  
for which some of the inequalities in the statement of Theorem~\ref{thm:degS}
 are actually equalities in case $d=n-1$.
  We will need some preliminary results, which include equivalences due
to Bihan and Soprunov (see \cite{BS17}), and
 the notion of an essential collection of subsets in Definition~\ref{def:adapted},
which was originally introduced in~\cite{Sturmfels94}.
We then have:

\begin{lemma} \label{lem:cond Alicia/Bihan} Let $B_1, \dots, B_{n-1}$ and $B_1', \dots, B_{n-1}'$ be
finite sets in $(\ZZ_{\ge 0})^{n-1}$ such that $B_i \subset B_i'$
for every $i \in [n-1]$. Then, the following statements are equivalent:
\begin{enumerate} [label={(\roman*)},itemindent=1em]
\item \label{lemitm:cond Alicia/Bihana} $\MV(B_1, \dots, B_{n-1})<\MV(B_1',
\dots, B_{n-1}')$.
\item \label{lemitm:cond Alicia/Bihanb} There exists a coherent collection of proper
faces $F=(F_1, \dots, F_{n-1})$ of the collection $B_1', \dots, B_{n-1}'$ such
that the collection $B_{1,F}', \dots, B_{{n-1},F}'$ is essential,
where $$B_{i,F}'=
\begin{cases} F_i \ \mbox{ if } B_i\cap F_i \ne \emptyset \\
B_i' \mbox{ if } B_i\cap F_i = \emptyset \end{cases}.$$
\item \label{lemitm:cond Alicia/Bihanc} There exists a coherent collection of proper faces
$F=(F_1, \dots, F_{n-1})$ of the collection $B_1', \dots, B_{n-1}'$ such that
the collection $\{B_i\cap F_i \ | \ B_i\cap F_i \ne \emptyset\}$
is either empty or essential.
\end{enumerate}
\end{lemma}
\begin{proof}
The equivalence between items \ref{lemitm:cond Alicia/Bihana} and
\ref{lemitm:cond Alicia/Bihanb} follows from \cite[Theorem
3.3]{BS17}. To see that item \ref{lemitm:cond Alicia/Bihana} is
equivalent to item \ref{lemitm:cond Alicia/Bihanc} we apply
Bernstein's Theorem. Let $g_1, \dots, g_{n-1}$ be generic
polynomials with support sets $B_1, \dots, B_{n-1}$. Write
$g_i(x)=\sum_{\alpha \in B_i}c_{i\alpha}x^\alpha$. Then $\MV(B_1,
\dots, B_{n-1})<\MV(B_1', \dots, B_{n-1}')$ if and only if there
exists a coherent collection of proper faces $F=(F_1, \dots,
F_{n-1})$ of $(B_1', \dots, B_{n-1}')$ such that the polynomials
$\{\sum_{\alpha \in B_i \cap
F_i}c_{i\alpha}x^\alpha\}_{i=1}^{n-1}$ have a common zero with all
non-zero coordinates. But, since $g_1, \dots, g_{n-1}$ are
generic, this is equivalent to the set $\{B_i\cap F_i \ | \
B_i\cap F_i \ne \emptyset\}$ being empty or essential.
\end{proof}

\begin{corollary} \label{cor:Bihan} Let $B_1,\dots,B_{n-1}$ be finite sets in
$(\ZZ_{\ge 0})^{n-1}$. Then, the following are equivalent:
\begin{enumerate}[label={(\roman*)},itemindent=1em]
\item \label{coritm:Bihana} $\MV(B_1, \dots, B_{n-1})=\vol(\conv(\cup_{i=1}^{n-1}B_i))$.
\item \label{coritm:Bihanb} For every $r \in [n-1]$ and every face $\mathcal{F}$ of
$\cup_{i=1}^{n-1}B_i$ of codimension $r$,   the cardinality of the set $\{i \ | \ B_i \cap
\mathcal{F} \neq \emptyset\}$ is at least  $n-r$.
\item \label{coritm:Bihanc} For every proper face $\mathcal F$ of $\cup_{i=1}^{n-1}B_i$ there exists
a nonempty subset $J \subset \{i \ | \ B_i \cap \mathcal{F} \neq
\emptyset\}$ such that  $\dim(\sum_{j \in J} (B_j \cap {\mathcal
F})) < |J|$.
\end{enumerate}
\end{corollary}
\begin{proof} The equivalence between items \ref{coritm:Bihana} and \ref{coritm:Bihanb}
follows from applying \cite[Corollary 3.7]{BS17} using the
polytope $\conv(\cup_{i=1}^{n-1}B_i)$. Item \ref{coritm:Bihanc}
follows from negation of item \ref{lemitm:cond Alicia/Bihanc} in
Lemma~\ref{lem:cond Alicia/Bihan}.
\end{proof}

\begin{example2} 
Recall that by Definition~\ref{def:AiJ}, for any any $I=\{j,k\}\subset[3]$ with 2 elements and $i$ the index in its
complement,
$d_I=\MV(A_j \cup A_0 \cup A_i, A_k \cup A_0 \cup A_i)$ and both supports are contained in the
union $\cup_{j=0}^3 A_i$.  It  can be checked that item~\ref{coritm:Bihanc}
in Corollary~\ref{cor:Bihan} holds and so
$d_I < \vol(\conv(\cup_{j=0}^3A_j))$ for all $I$.
\end{example2}

The following example shows that the conditions given
in items~\ref{lemitm:cond
Alicia/Bihanb} and~\ref{lemitm:cond Alicia/Bihanc} in
Lemma~\ref{lem:cond Alicia/Bihan} are different.

\begin{example} \label{ex:different}
Consider the finite sets in $(\ZZ_{\ge 0})^3$
$$\begin{array}{lcl}
B_1  = \{ (0,0,0), (1,0,0)\},          &              &
B_1' = \{ (0,0,0), (2,0,0), (0,2,0), (0,0,2)\}, \\
B_2  = \{ (0,1,0), (1,1,0), (0,2,0)\}, &              &
B_2' = \{ (0,0,0), (2,0,0), (0,2,0)\}, \\
B_3  = \{ (0,1,0), (1,1,0), (0,2,0)\}  & \mbox{ and } &
B_3' = \{(0,0,0), (2,0,0), (0,2,0)\}.
\end{array}$$
\begin{center}
\tdplotsetmaincoords{70}{110}
\begin{tikzpicture}[tdplot_main_coords,scale=0.4]
\draw[thick,->] (0,0,0) -- (3,0,0) node[anchor=north]{$x$};
\draw[thick,->] (0,0,0) -- (0,3,0) node[anchor=north]{$y$};
\draw[thick,->] (0,0,0) -- (0,0,3) node[anchor=south]{$z$};
\draw[thick,cyan] (0,0,0) -- (1,0,0); \draw[fill] (0,0,0) circle
(2pt) (1,0,0) circle (2pt); \draw (0,1,1) node {$B_1$};
\end{tikzpicture}
\tdplotsetmaincoords{70}{110}
\begin{tikzpicture}[tdplot_main_coords,scale=0.4]
\draw[thick,->] (0,0,0) -- (3,0,0) node[anchor=north]{$x$};
\draw[thick,->] (0,0,0) -- (0,3,0) node[anchor=north]{$y$};
\draw[thick,->] (0,0,0) -- (0,0,3) node[anchor=south]{$z$};
\draw[fill, violet] (0,1,0) -- (1,1,0) -- (0,2,0) -- (0,1,0);
\draw[fill] (0,1,0) circle (2pt) (1,1,0) circle (2pt) (0,2,0)
circle (2pt) (0,1,0) circle (2pt); \draw (0,1,1) node {$B_2$};
\end{tikzpicture}
\tdplotsetmaincoords{70}{110}
\begin{tikzpicture}[tdplot_main_coords,scale=0.4]
\draw[thick,->] (0,0,0) -- (3,0,0) node[anchor=north]{$x$};
\draw[thick,->] (0,0,0) -- (0,3,0) node[anchor=north]{$y$};
\draw[thick,->] (0,0,0) -- (0,0,3) node[anchor=south]{$z$};
\draw[fill, orange] (0,1,0) -- (1,1,0) -- (0,2,0) -- (0,1,0);
\draw[fill] (0,1,0) circle (2pt) (1,1,0) circle (2pt) (0,2,0)
circle (2pt) (0,1,0) circle (2pt); \draw (0,1,1) node {$B_3$};
\end{tikzpicture}
\tdplotsetmaincoords{70}{110}
\begin{tikzpicture}[tdplot_main_coords,scale=0.4]
\draw[thick,->] (0,0,0) -- (3,0,0) node[anchor=north]{$x$};
\draw[thick,->] (0,0,0) -- (0,3,0) node[anchor=north]{$y$};
\draw[thick,->] (0,0,0) -- (0,0,3) node[anchor=south]{$z$};
\draw[fill, cyan] (0,0,0) -- (2,0,0) -- (0,2,0) -- (0,0,0);
\draw[fill, cyan, opacity=0.5] (0,0,2) -- (2,0,0) -- (0,2,0) --
(0,0,2); \draw[fill] (0,0,0) circle (2pt) (2,0,0) circle (2pt)
(0,2,0) circle (2pt) (0,0,2) circle (2pt); \draw (0,1.5,1.5) node
{$B_1'$};
\end{tikzpicture}
\tdplotsetmaincoords{70}{110}
\begin{tikzpicture}[tdplot_main_coords,scale=0.4]
\draw[thick,->] (0,0,0) -- (3,0,0) node[anchor=north]{$x$};
\draw[thick,->] (0,0,0) -- (0,3,0) node[anchor=north]{$y$};
\draw[thick,->] (0,0,0) -- (0,0,3) node[anchor=south]{$z$};
\draw[fill, violet] (0,0,0) -- (2,0,0) -- (0,2,0) -- (0,0,0);
\draw[fill]  (0,0,0) circle (2pt) (2,0,0) circle (2pt) (0,2,0)
circle (2pt) (0,0,0) circle (2pt); \draw (0,1.5,1.5) node
{$B_2'$};
\end{tikzpicture}
\tdplotsetmaincoords{70}{110}
\begin{tikzpicture}[tdplot_main_coords,scale=0.4]
\draw[thick,->] (0,0,0) -- (3,0,0) node[anchor=north]{$x$};
\draw[thick,->] (0,0,0) -- (0,3,0) node[anchor=north]{$y$};
\draw[thick,->] (0,0,0) -- (0,0,3) node[anchor=south]{$z$};
\draw[fill, orange] (0,0,0) -- (2,0,0) -- (0,2,0) -- (0,0,0);
\draw[fill]  (0,0,0) circle (2pt) (2,0,0) circle (2pt) (0,2,0)
circle (2pt) (0,0,0) circle (2pt); \draw (0,1.5,1.5) node
{$B_3'$};
\end{tikzpicture}
\end{center}

Clearly $0 = \MV(B_1, B_2, B_3) < \MV(B_1', B_2', B_3')$. If we
take the coherent collection of faces of $B_1', B_2', B_3'$
associated to the interior normal vector $(1,1,0)$, the collection
from item \ref{lemitm:cond Alicia/Bihanb} in Lemma~\ref{lem:cond Alicia/Bihan} is essential. However, the
collection from item \ref{lemitm:cond Alicia/Bihanc} is not. On
the other hand, if we consider the coherent collection of faces
associated to the interior normal vector $(0,1,1)$, the collection
from item \ref{lemitm:cond Alicia/Bihanc} Lemma~\ref{lem:cond Alicia/Bihan} is essential but the
collection from item \ref{lemitm:cond Alicia/Bihanb} is not.

Therefore,  it is not straightforward to prove in
a direct way the equivalence of items~\ref{lemitm:cond
Alicia/Bihanb} and~\ref{lemitm:cond Alicia/Bihanc} in
Lemma~\ref{lem:cond Alicia/Bihan}. It would be interesting to see
a direct combinatorial proof without going through item~\ref{lemitm:cond Alicia/Bihana}.
\end{example}

The following theorem provides a necessary and
sufficient condition for the equality $d_I =
\vol(\conv(\cup_{j=0}^n A_j))$ for some $ I \subset [n]_0$ 
with $|I| = n-1 $. We will denote by $I^c$ the complement of $I$ in $[n]_0$.

\begin{theorem}\label{thm:vol=mindI} Let $A_0,\dots, A_n$ be finite sets
in $(\ZZ_{\ge 0})^{n-1}$. Fix $I\subset [n]_0$ with $|I|=n-1$.
The following statements are equivalent:
\begin{enumerate}[label={(\roman*)},itemindent=1em]
\item \label{thmitm:vol=mindIa} For every $r \in [n-1]$ and
every face $\mathcal{F}$ of $\cup_{j=0}^n A_j$ of codimension $r$
and every $J \subset [n]_0$ such that $|J|=r+2$ that contains 
$I^c$, we have $(\cup_{j\in J}A_j)\cap \mathcal{F} \ne \emptyset$.
\item \label{thmitm:vol=mindIb} $\vol(\conv(\cup_{j=0}^n A_j))= d_I$.
\end{enumerate}
 \end{theorem}
 \begin{proof} For every set $J\subset [n]_0$ that
 contains $I^c$, we can define $\widetilde{J}= J \backslash I^c \subset
 I$.  Then item \ref{thmitm:vol=mindIa} happens if and only if for every face $\mathcal{F}$
 of $\cup_{j=0}^n A_j$ of codimension $r>0$ and every $\widetilde{J} \subset
 I$ with $|\widetilde{J}|=r$, $(\cup_{j\in \widetilde{J}}(A_j\cup (\cup_{i \in I^c}A_i))\cap \mathcal{F}
\ne \emptyset$. Using item \ref{coritm:Bihanb} of
Corollary~\ref{cor:Bihan}, this is equivalent to item
\ref{thmitm:vol=mindIb}.
 \end{proof}

The next theorem provides a sufficient condition for
the equality $\deg(S) \cdot \deg(\ff)=d_{I}$ for some $ I \subset
[n]_0$ with $|I| = n-1$. 

\begin{theorem} \label{thm:degH=mindI} Let $A_0,\dots, A_n$ be finite
sets in $(\ZZ_{\ge 0})^{n-1}$. Let $f_0,\dots, f_n$ be generic
polynomials with support $A_0, \dots, A_n$ and coefficients in $\K$ such that the rational map 
$\ff$ from~\eqref{eq:def3} is generically finite.
Assume $I \subset [n]_0$ with $|I|=n-1$ satisfies that for every
coherent collection of proper faces $F$ of $(A_i\cup (\cup_{j \in
I^c}A_j))_{i \in I}$, there exists $J$ nonempty subset of
$\mathcal{P}_F = \{i \in I \ | \ A_i\cap F_i \ne \emptyset\}$ such
that $$ \mbox{dim}(\sum_{i\in J}A_i\cap F_i)<|J|.$$ Then,
deg$(S) \cdot \deg(\ff) = d_{I} =
 \MV(A_{i_1}, \dots, A_{i_{n-1}})$. 
\end{theorem}
\begin{proof}
Denote $I=\{i_1, \dots, i_{n-1}\}$ and $I^c=\{i_0,i_n\}$.
 As in the proof of Theorem~\ref{thm:degS}, take generic coefficients
$\{\mu_{j,I},\nu_{j,I}\}_{j=1}^{n-1}$ so that the system 
$h_{j,I} = f_{i_j}+\mu_{j,I}\, f_{i_0} + \nu_{j,I}\, f_{i_n}=0$ for $ j \in [n-1]$ has $\deg(S)\, \deg(\ff)$ common isolated zeros,
all of them in the open set  $\{x \in 
(\K^*)^{n-1}\ | \ \prod_{j=0}^nf_j(x)\ne 0\}$.

Consider the homotopy $\{f_{i_j}+t\, \mu_{j,I}\, f_{i_0} +
t\, \nu_{j,I}\, f_{i_n}\}_{j=1}^{n-1}$. For all but finitely many values of $t \in \K$,
the system given by these polynomials has $\deg(S)\cdot \deg(\ff)$
isolated common zeros with all non-zero coordinates, and
$\MV(A_{i_1}, \dots, A_{i_{n-1}})$ when $t=0$. Then, using
Theorem~\ref{thm:degS}, $\MV(A_{i_1}, \dots, A_{i_{n-1}})\le
\deg(S)\cdot \deg(\ff) \le d_{I} = \MV(A_{i_1}\cup A_{i_0} \cup
A_{i_n}, \dots, A_{i_{n-1}}\cup A_{i_0} \cup A_{i_n})$ common
zeros with all non-zero coordinates. By Lemma~\ref{lem:cond
Alicia/Bihan}, $\MV(A_{i_1}, \dots, A_{i_{n-1}})= d_{I}.$
\end{proof}

\begin{example} Let $d \in \NN$ be an even number. Consider the
hypersurface $S$ parameterized by the rational map $\ff$ as in~\eqref{eq:def3},
where $f_0, \dots, f_3$ are generic polynomials with two variables,
coefficients in $\K$ and supports {\small$$A_0 =
\{(0,0),(1,0),(0,1), (d/2,d/2)\} = A_3, \ A_1 = \{(0,0), (1,0),
(0,d)\}, \ A_2=\{(0,0), (d,0), (0,1)\}.$$} Then the conditions of
Theorem~\ref{thm:degH=mindI} are satisfied for $I = \{1,2\}$ and
we have that $\deg(\ff)=1$, so the degree of the hypersurface  equals
$\deg(S) = d^2 =
MV(A_0\cup A_1 \cup A_3, A_0\cup A_2 \cup A_3)$.
\end{example}

\section{Orders of rational hypersurfaces under sparse parametrizations}\label{sec:order}

In this section, we assume  that $d=n-1$ and $S$ is a hypersurface in $(\K^*)^n$.  
We will study its order at the origin $\mbox{ord}_0(S)$, which is defined as follows.
 Let $H$ be a  (reduced) polynomial such that $S=(H=0)$. Then,
$$\mbox{ord}_0(S) = \max\{m\in \ZZ_{\ge 0} \ | \ \partial^\beta(H)(0)=0,
 \text{ for all } \beta\in (\ZZ_{\ge 0})^n, |\beta| < m\}.$$
 We compute ${\rm ord}(S)$ in Theorem~\ref{thm:ord} under an additional condition on the
 family of supports.  As we remarked in the Introduction, depending 
on the relative positions of the supports $A_0,\dots,A_n$ it could happen   that ${\rm ord}(S) > 1$ 
(that is, the origin is a singular point of the closure of $S$ in $\K^n$) even for generic coefficients, 
We end with two examples where our proofs can be extended to other cases not falling into
the hypotheses of Theorem~\ref{thm:ord}.

%

Before stating Theorem~\ref{thm:ord}, we need to introduce a
definition. 

\begin{definition}\label{def:tame}
 Let $A_0, \dots, A_n$ be a family of finite sets in $(\ZZ_{\ge 0})^{n-1}$.
Given $\alpha \in \bQ^{n-1}$, we call $J_\alpha = \{ j \in [n]_0 \ | \ m_\alpha(A_j) < \max_{i \in
[n]_0}\{m_\alpha(A_i)\}\}$. We say that the family is tame if for
any $\alpha$ such that $J_\alpha \ne \emptyset$ and $|J_\alpha|\ne
n$, the family $\{\mbox{face}_\alpha(A_j)\}_{j \in J_\alpha}$ in
$(\ZZ_{\ge 0})^{n-1}$ is not essential (see
Definition~\ref{def:adapted}). 
\end{definition}
Note that in case all $A_i$ are equal for all $i \in [n]_0$ or if
all $A_i$ are equal for $i \in [n]$ and $A_0$ consists of a single
point (thus $f_i/f_0$ are Laurent polynomials with the same
support, the case studied in~\cite{SY94}), the family is always
tame.

Recall that by Theorem~\ref{thm:degS}  the degree of $S$ for generic coefficients satisfies
$\deg(S) \cdot \deg(\ff) \, =\,  \MV({A}_0^{2}, \dots, {A}_n^{2})$,     where
the  notation $A_i^2 \subset \ZZ^{n+1}$  is given in Definition~\ref{def:widehat}.

\begin{theorem}\label{thm:ord}
Let $A_0, \dots, A_n$ be a tame family of finite sets in
$(\ZZ_{\ge 0})^{n-1}$. Let $f_0,\dots, f_n$ be generic polynomials
with respective supports $A_0,\dots, A_n $ and coefficients in $\K$ such that
the map $\ff$ from~\eqref{eq:def3} is generically finite and parameterizes a 
hypersurface $S$ in $\K^n$.
Then,
$$\mbox{ord}_0(S) \cdot \deg(\ff)=  \MV({A}_0^{2}, \dots, {A}_n^{2}) -
MV({A}_1^1, \dots{A}_n^1) \ge \MV({A}_0^{2}, \dots, {A}_n^{2})
- \min_{j \in [n]}\{\MV(\{A_i\cup A_j\}_{i \in [n] \atop i
\ne j}) \}.$$
\end{theorem}

\begin{proof}
Consider $\ell_i=a_{i,1} y_1 + \cdots + a_{i,n}y_n$ generic linear
forms with coefficients in $\K$, and generic constants $a_{i,0}\in
\K^*$ for all $i \in [n-1]$. As before, consider the open set $U= \{x
\in (\KK^*)^{n-1}\ | \ \prod_{j=0}^nf_j(x)\ne 0\}$. We can assume
that $H,\ell_1-a_{1,0}\varepsilon, \dots,
\ell_{n-1}-a_{n-1,0}\varepsilon$ have $D=\MV({A}_0^{2}, \dots, {A}_n^{2})/\deg(\ff)$
 common roots, and that they all lie in $\ff_\KK(U)$.

To compute the order at the origin of $H$, we will consider the
line in $\KK^n$ defined by the generic affine linear forms
$\ell_1-a_{1,0}\varepsilon, \dots,
\ell_{n-1}-a_{n-1,0}\varepsilon$. Notice that this line can be
parameterized as
$$\{x\in\KK^n \ | \ \ell_1(x) = a_{1,0}\varepsilon, \dots,
\ell_{n-1}(x) = a_{n-1,0}\varepsilon\}=\{\lambda\, u + \varepsilon
\,v \ | \ \lambda \in \KK\},$$ with $u, v\in \K^n$. The common
zeros of the polynomials $H,\ell_1-a_{1,0}\varepsilon, \dots,
\ell_{n-1}-a_{n-1,0}\varepsilon$ in $ \ff_{\KK}(U)$ are exactly
the points
$$\omega_k=\left(
  \frac{f_1(\sigma_k(\varepsilon))}{f_0(\sigma_k(\varepsilon))},
  \dots,
  \frac{f_n(\sigma_k(\varepsilon))}{f_0(\sigma_k(\varepsilon))}\right) =
  \lambda_k\, u + v\,\varepsilon , \quad   k \in [D],  $$
for which $H(\lambda_k\, u + v\,\varepsilon)=0$. If we consider
the polynomial $H(\lambda\, U + \varepsilon \,V)\in
\K[U,V,\varepsilon][\lambda]$, we have
\begin{equation}\label{eq:H}
 H(\lambda\, U + \varepsilon \,V)=H_{0}(U) \lambda^D +
 H_1(U,V,\varepsilon) \lambda^{D-1} + \cdots +
 H_D(V,\varepsilon),
\end{equation}
where for all $j\in [D]_0$, $H_j(U,V, \varepsilon)$ is a
polynomial of degree  at most $j$ in $\varepsilon$.

By genericity of $\ell_i$ and $a_{i,0}$, we can assume that $u,
v\in (\K^*)^n$, $H_{0}(u)\neq 0$ and $H_D(v,\varepsilon) =
H(v\,\varepsilon) \neq 0$. Note that $\val(H_0(u))=0$ and
$\val(H_j(u,v,\varepsilon))\ge 0$ for all $ j \in [D-1]_0$.
Hence, if we consider the coefficients of $H(u\lambda +
v\varepsilon) = \sum_{i=0}^DH_{D-i}(u,v,\varepsilon)\lambda^i\in
\KK[\lambda]$ as elementary symmetric functions of the roots
$\{\lambda_k\}_{k=1}^D\subset \KK$, we can easily see that
$\val(\lambda_k) \ge 0$ for all $k \in [D]$. Indeed, assume
that $\val(\lambda_k) < 0$ for some $k$ and consider the non-empty
set $J=\{k \in [D] \ | \ \val(\lambda_k)<0\}$. Then,
the elementary symmetric function $H_{\#J}(u,v,\varepsilon)$ has
the same negative valuation as $H_0(u)\prod_{k \in J}\lambda_k$,
which is a contradiction.
Moreover, the minimum power of $\lambda$ with a coefficient of
valuation $0$ is $D - \#\{k \ | \ \val(\lambda_k)=0\}.$
On the other hand, if we now consider $H$ as a sum of homogeneous
terms of degree $\mbox{ord}_0(S)$ to $D$ and we evaluate it at
$u\lambda + v\varepsilon$, the minimum power of $\lambda$ with a
coefficient of valuation $0$ is $\mbox{ord}_0(S)$. Therefore,
\begin{equation}\label{eq:resta}
\mbox{ord}_0(S) = D - \#\{k \ | \ \val(\lambda_k)=0\}.
\end{equation}

We will focus in what follows on bounding the number of points
$\{\omega_k\}_{k=1}^D$ for which $\val(\lambda_k)=0$. Let
$\omega_{k}=\left(\frac{f_1(\sigma_{k})}{f_0(\sigma_{k})}, \dots,
\frac{f_n(\sigma_{k})}{f_0(\sigma_{k})}\right)$ be one of them.
Since $m = \val(f_i(\sigma_k))$ is fixed for all $i \in [n]_0$ and
the supports are tame, let us prove that this implies that if
$\sigma_{k}= b_{k}\varepsilon^\alpha + \hot(\varepsilon)$ with
$b_{k} \in (\K^*)^n$, then necessarily $\alpha=0$.
Assume first that $\alpha\neq 0$ and recall the notation
in~\eqref{eq:4.1}: $m_{\alpha}(A_j)=\min_{p\in A_j}\<\alpha,p\>$.
Since $\init_\alpha(f_1), \dots, \init_\alpha(f_n)$ are generic,
they do not all vanish on $b_k$. Also, $m = m_\alpha(A_j)$ if and
only if $\init_{\alpha}(f_j)(b_k) \ne 0 $ for all $ j \in [n]_0.$
However, since the supports are tame, $\{\init_{\alpha}(f_j)\}_{j
\in J_\alpha}$ can only have common zeros with all non-zero
coordinates when $J_\alpha = \emptyset$, and so
$\init_{\alpha}(f_j)(b_k) \ne 0 $ for all $ j \in [n]_0.$ As
$a_{i,1}f_1(\sigma_k) + \dots + a_{i,n}f_n(\sigma_k) -
a_{i,0}f_0(\sigma_k)\varepsilon = 0 \mbox{ for all } i \in [n-1]$, 
$b_k\in (\K^*)^{n-1}$ is a common zero of the polynomials
$\ell_i(\init_{\alpha}(f_1(x)), \dots, \init_{\alpha}(f_n(x)))\in
\K[x_1, \dots, x_{n-1}]$  for all  $i\in [n-1]$.
The convex hull of the support of any linear combination of
$\init_{\alpha}(f_1(x)),$ $ \dots, \init_{\alpha}(f_n(x))$ has
dimension at most $n-2$ because all supports are contained in
the hyperplane with equation $\langle \alpha, p \rangle =m$.
Then, the variety $W_\alpha \subset (\K^*)^n$ parameterized for all $x \in U$ by $(\init_{\alpha}(f_1(x)),
\dots, \init_{\alpha}(f_n(x))$ has dimension at most $n-2$. 
The space of lines from  a point in $W_\alpha$ through the origin
has then codimension at least one in the space of lines in $\K^n$
through the origin. Thus, the intersection of $W_\alpha$
with a generic line  $\ell_1=0, \dots,
\ell_{n-1}=0$ is empty. Therefore, $\alpha = 0$.

Consider then $\sigma_{k}= b_{k} + \hot(\varepsilon)$ with $b_{k}
\in (\K^*)^{n-1}$. 
The polynomials 
$$g_i =\ell_i(f_1(x), \dots, f_n(x)), \quad i\in [n-1],$$
have finitely many common zeros in $(\K^*)^{n-1}$, $b_k$ being
one of those zeros. As in the proof of Theorem~\ref{thm:degS}, by taking
linear combinations, we can tranform them  into
polynomials of the form $f_1 +\mu_1 f_n, \dots,  f_{n-1} + \mu_{n-1} f_n$  with generic
coefficients $\mu_1, \dots, \mu_{n-1}$, and thus its
number of common zeros  is bounded above by $\min_{j \in [n]}\{\MV(\{A_i\cup
A_j\}_{i \ne j,0})\}$. Moreover, one can take an
extra variable $x_n$ and define an associated generic system
${h}_{1} = \dots = h_n= 0$  where
$${h}_i = f_i + \mu_i \, x_n\in \K[x_1, \dots, x_n] \mbox{ for all } i \in [n-1], \mbox{ and } {h}_n = f_n - x_n.$$
By the genericity of their coefficients, $f_1, \dots, f_n$ do not have common zeros, and
hence common zeros of $g_1, \dots, g_{n-1}$ 
in $(\K^*)^{n-1}$ correspond to common zeros of the polynomials $h_1, \dots, h_n$
in $(\K^*)^n$.
By Bernstein's Theorem, the number of common roots of $g_1, \dots g_{n-1}$ equals
$\MV(\{{A}^1_i\}\}_{i =1}^n)$.
Since $U$ is smooth, \cite[Corollary 6.7.2]{Jouanolou83} ensures
that the Jacobian matrix of $g_1, \dots, g_{n-1}$ has nonzero
determinant when evaluated at a common zero $b\in U \subset
(\K^*)^{n-1}$. This implies that the Jacobian matrix of
$\ell_i(f_1(x), \dots, f_n(x))- a_{i,0}f_0(x)\varepsilon$ has
valuation zero when evaluated at the common roots $b\in
(\K^*)^{n-1}$ of $g_1, \dots, g_{n-1}$ and, by Hensel's Lemma,
each of these common zeros $b$ can be lifted to a unique
common zero  $\sigma$ of $a_{i,1}f_1 + \dots + a_{i,n}f_n -
a_{i,0}f_0\varepsilon$, for all $i \in [n-1]$.
We deduce that the number of $\lambda_k$ with $\val(\lambda_k) =
0$ is $\deg(\ff)^{-1}\cdot\MV(\{{A}^1_i\}\}_{i =1}^n)$ and the result follows from
\eqref{eq:resta}.
\end{proof}

The following proposition
presents a sufficient condition for the inequality in the statement
Theorem~\ref{thm:ord} to be an equality and the order can be
computed in an easier way.

\begin{proposition}\label{prop:valuation} Under the hypotheses and notation
of Theorem~\ref{thm:ord}, if 
there exists $j_0 \in [n]$ such that for all coherent
collection of proper faces $F$ of $(A_i\cup A_{j_0})_{i\neq j_0}$,
there exists a nonempty subset $J$ of 
$\mathcal{P}_F = \{i\in [n]\backslash \{ j_0 \} \ | \ A_i\cap F_i \ne \emptyset\}$
such that $\mbox{dim}(\sum_{i\in J}A_i\cap F_i)<|J|$, then
\begin{equation}\label{eq:diff}
{\rm ord}_0(S) \cdot \deg(\ff) = \MV({A}_0^{2}, \dots, {A}_n^{2}) -
\MV(A_1, \dots, A_{j_0-1}, A_{j_0+1}, \dots, A_n ).
\end{equation}
In particular,
\begin{equation}\label{eq:3}
MV({A}_1^1, \dots{A}_n^1) =  \MV(\{A_i\cup A_{j_0}\}_{i \in [n] \atop i \neq j_0}) =  \MV(A_1, \dots, A_{j_0-1}, A_{j_0+1}, \dots, A_n).
\end{equation}
\end{proposition}

\begin{proof} 
Denote $\MV_{j_0}= \MV(A_1, \dots, A_{j_0-1}, A_{j_0+1}, \dots, A_n)$.
As the polynomials $f_0, \dots, f_n$ are generic with respect to their supports, we have by Theorem~\ref{thm:degS} that $\deg(S) \cdot \deg(\ff) = 
\MV({A}_0^{2}, \dots, {A}_n^{2})$. 
Because the supports are tame, we saw in the proof of Theorem \ref{thm:ord}  that 
$ord_0(S)\cdot \deg(\ff) = \MV({A}_0^{2}, \dots, {A}_n^{2}) - \#\{k \ | \ \val(\lambda_k)=0\}\cdot \deg(\ff)$ and that
we can compute $N = \#\{k \ | \ \val(\lambda_k)=0\}\cdot \deg(\ff)$
as the number of common zeros in $(\K^*)^{n-1}$ of $\ell_i(f_1(x), \dots, f_n(x))$ for $i\in[n-1]$.
Taking linear combinations, this system is equivalent to 
$f_i + \mu_if_{j_0}=0 \ \mbox{ for all } i \in [n], \ i \neq j_0$.
Thus, $N$ is bounded above by $\MV(\{A_i\cup A_{j_0}\}_{i \in [n] \atop i \neq j_0})$.
Now, as in the proof of Theorem~\ref{thm:degH=mindI}, we introduce a new variable $t$ and consider the homotopy:
\begin{equation}\label{eq:homotOrd} 
f_i + t\mu_if_{j_0}=0 \ \mbox{ for all } i \in [n], \ i \neq j_0.
\end{equation}
For almost all $t$, its number of common zeros in $(\K^*)^{n-1}$ is $N$.
As $f_1, \dots, f_n$ are generic polynomials, they have $MV(A_1, \dots, A_{j_0-1}, A_{j_0+1}, \dots, A_n)$ common zeros with all non-zero coordinates. { Since these zeros are isolated, each is in a germ of curve of zeros of \eqref{eq:homotOrd} not contained in $\{t=0\}$.} This implies that $MV(A_1, \dots, A_{j_0-1}, A_{j_0+1}, \dots, A_n) \leq N$.
Therefore,
\[ \MV_{j_0} \le N \le \MV(A_1\cup A_{j_0}, \dots, A_{j_0-1}\cup A_{j_0},
A_{j_0+1}\cup A_{j_0}, \dots, A_n\cup A_{j_0}) .\]
But now,   it follows from Lemma~\ref{lem:cond Alicia/Bihan} that $\MV(A_1,\dots, A_{j_0-1},A_{j_0+1},\dots, A_n) = \MV_{j_0}$,
and hence 
${\rm ord}_0(S) \, \deg(\ff) = MV(A_0^2, \dots, A_n^2) - MV_{j_0}$, as claimed. The remaining equality in~\eqref{eq:3} follows
from the computation of ${\rm ord}_0(S)$ in the statement of Theorem~\ref{thm:ord}.
\end{proof}

\begin{example}\label{ex:genDegree&Order} Consider again the application
$\mathbf{f}$ given by the generic polynomials from
Example~\ref{ex:generic}. 
Since $f_0, f_1, f_2, f_3$ are generic, the conditions
 in Proposition~\ref{prop:valuation} are realized for any
$j_0 \in [3]$ and we deduce that ${\rm ord}_0(S)=4$, as
we mentioned, because $\deg(\ff)=1$.
\end{example}

\begin{remark} While the hypothesis in the statement of Theorem~\ref{thm:ord}
 that the supports are tame is
sufficient, it is not always necessary. For example, consider the
family of lattice sets {\small $A_0 = \{(0,0,0)\}$, $A_1 = \{(2,2,0),
(1,1,1)\}$, $A_2 = \{(1,1,0), (1,2,1)\}$, $A_3 = \{(2,1,0),(2,0,1)\}$,
$A_4 = \{(2,1,0),(1,0,0)\}$}. It is easy to check that this family of supports is not tame: for $\alpha =
(-1,1,1)$, the set $J_\alpha$ from Definition \ref{def:tame} is $J_\alpha = \{3,4\}$. However, the family $\{\mbox{face}_\alpha(A_3), \mbox{face}_\alpha(A_4)\}$ is essential.
Even in this case, for generic $f_0, f_1,
f_2, f_3, f_4$ with those supports, we still have (as in the statement of Theorem  \ref{thm:ord}) that
$$\mbox{ord}_0(S) \cdot \deg(\ff) =  \MV({A}_0^{2}, \dots, {A}_4^{2}) -
MV({A}_1^1, \dots{A}_4^1) \ge \MV({A}_0^{2}, \dots, {A}_4^{2})
- \min_{j \in [4]}\{\MV(\{A_i\cup A_j\}_{i \in [4] \atop i
\ne j}) \}.$$ Moreover, the inequality above is an equality. It can be computed  that for generic polynomials, 
$\deg(\ff)=1$, $\deg(S)=9$,
${\rm ord}_0(S)=4$, and
$\MV(\{{A}^1_i\}_{i=1}^4) = \min_{i \in [4]}\{\MV(\{A_j \cup A_i\}_{j \ne i})\} =
5$.
\end{remark}

\begin{example}
Consider again Example~\ref{ex:nongeneric} with non-generic coefficients. 
As in Example~\ref{ex:genDegree&Order}, by applying Bernstein's Theorem
to the polynomial systems in the proof of Theorem~\ref{thm:degS},
we can see that all inequalities in
Theorem~\ref{thm:degS} 
are actually equalities and the degree of the surface
$S$ is $9 =
\vol(\conv(\cup_{j=0}^3A_j))$.
The polytope $\conv(\cup_{j=1}^{3} A_j)$ with vertices
$(2,0),(0,2),(3,0),(0,3)\in (\ZZ_{\ge 0})^2$ has volume
$\vol(\conv(\cup_{j=1}^{3} A_j)) = 5$. Theorem~\ref{thm:ord}
predicts the lower bound $4$ for generic systems with those
supports, which is also true in this case. To see that, we can
follow the steps in its proof, and first verify that there are no
$\sigma_k=b_k\varepsilon^\alpha + \mbox{h.o.t.}(\varepsilon)\in
\CC\{\{\varepsilon^\RR\}\}$ with $\alpha \ne 0$ and
$\val(\lambda_k) = 0$. Then, we can check that the polynomials
$\ell_i(f_1,f_2,f_3)$ for $i\in \{1,2\}$ have exactly $5$ common
zeros in $(\K^*)^2$ by means of the genericity conditions in
Bernstein's Theorem. Finally, applying \cite[Corollary
6.7.2]{Jouanolou83} and Hensel's Lemma, the order at the origin is
exactly $9-5=4$. 
\end{example}

\end{document}